\documentclass[12pt,final]{amsart}
\usepackage[margin=1in,includeheadfoot]{geometry}
\usepackage{amsmath}
\usepackage{slashed}
\usepackage{graphicx}

\usepackage{mathrsfs}
\usepackage{txfonts}
\usepackage{amssymb,dsfont}
\usepackage{enumerate}
\usepackage{fancyhdr}
\usepackage{aliascnt}
\usepackage{pdfsync}
\usepackage{hyperref}

 \usepackage{color}

\def\beq{\begin{equation}}
\def\eeq{\end{equation}}
\def\ba{\begin{array}}
\def\ea{\end{array}}
\def\S{\mathbb S}
\def\R{\mathbb R}
\def\N{\mathbb N}

\def\la{\langle}
\def\ra{\rangle}

\def \ds{\displaystyle}
\def \vs{\vspace*{0.1cm}}

\def\slashii#1{\setbox0=\hbox{$#1$}             
   \dimen0=\wd0                                 
   \setbox1=\hbox{\sl/} \dimen1=\wd1            
   \ifdim\dimen0>\dimen1                        
      \rlap{\hbox to \dimen0{\hfil\sl/\hfil}}   
      #1                                        
   \else                                        
      \rlap{\hbox to \dimen1{\hfil$#1$\hfil}}   
      \hbox{\sl/}                               
   \fi}                                         %
\def\slashiii#1{\setbox0=\hbox{$#1$}#1\hskip-\wd0\hbox to\wd0{\hss\sl/\/\hss}}
\newcommand{\C}{{\mathbf C}}

\newtheorem{thm}{Theorem}[section]
\newtheorem{lm}[thm]{Lemma}
\newtheorem{prop}[thm]{Proposition}

\newtheorem{conj}[thm]{Conjecture}

\theoremstyle{definition}

\theoremstyle{remark}

\begin{document}

\title{The Super-Toda System and Bubbling of Spinors}

\author[Jost]{J\"urgen Jost}
\address{Max Planck Institute for Mathematics in the Sciences\\ Inselstrasse 22\\ 04103 Leipzig, Germany}
\email{jost@mis.mpg.de}

\author[Zhou]{Chunqin Zhou}
\address{School of Mathematical Sciences, Shanghai Jiao Tong University\\ 800 Dongchuan Road \\ Shanghai, 200240 \\China}
\email{cqzhou@sjtu.edu.cn}

\author[Zhu]{Miaomiao Zhu}
\address{School of Mathematical Sciences, Shanghai Jiao Tong University\\ 800 Dongchuan Road \\ Shanghai, 200240 \\China}
\email{mizhu@sjtu.edu.cn}

\thanks{The second author was supported partially by NSFC of China (No. 11271253). The third author was supported in part by NSFC of China (No. 11601325).}

\date{\today}
\begin{abstract}
We introduce the super-Toda system on Riemann surfaces and study the blow-up analysis for a sequence of solutions to the super-Toda system on a closed Riemann surface with uniformly bounded energy. In particular, we show the energy identities for the spinor parts of a blow-up sequence of solutions for which there are possibly four types of bubbling solutions, namely, finite energy solutions of the  super-Liouville equation or the super-Toda system defined on $\R^2$ or on $\R^2\setminus\{0\}$. This is achieved by showing some new energy gap results for the spinor parts of these four types of bubbling solutions.
\end{abstract}


\maketitle

\section{Introduction}

The supersymmetric two-dimensional Toda lattice was introduced by Olshanetsky \cite{O}. It is a natural generalization or combination  of  the supersymmetric extension of the Liouville equation \cite{P1,P2} on one hand and the standard Toda lattice on the other hand. Olshanetsky \cite{O} studied it from the perspective of Lie algebras and integrable systems. In this paper, we consider a version from the perspective of geometric analysis. Our approach is more in some sense narrower  than that of \cite{O}, because we only consider the Lie group $SU(N+1)$ for $N=2$, but here, we are interested in the analytical and not so much in the algebraic aspects of the theory, and for that, the choice of Lie group is not so important. Our setting is also different from that of \cite{O}, insofar as we set up the theory with commuting variables only, instead of considering anticommuting spinor fields. This offers the crucial advantage that we can apply methods from nonlinear partial differential equations (see \eqref{0-4} or \eqref{equ-1} below) and the calculus of variations (see \eqref{cv}), and exploring those connections is the main purpose of this work. Remarkably, and perhaps surprisingly, the symmetries of the system are not much affected. We lose the supersymmetry between the commuting and the anticommuting variables, but, for instance, the more important conformal symmetry is preserved. Therefore, here, we do not discuss the extensive physics literature any further, but let us mention at least  some  references  about super-Liouville
equations, \cite{ARS, FH}, that provide some further background and motivation. (The super-Liouville equations can be seen as super-Toda equations for $SU(N+1)$ for $N=1$.) Thus, we concentrate on the analytical aspects, and in that direction, our main new contribution seems to be a bubbling analysis for spinors.

The two-dimensional field equations of mathematical physics provide a rich source of deep analytical problems that have lead to new methods and insights that have stimulated much subsequent work, see for instance \cite{CY} or \cite{T}. The present work aims to proceed further in this direction, and, in particular, explore phenomena that arise from coupling different fields in a ``physically correct'' manner.

Let us now describe the system in more technical terms.
Let $(M,g)$ be a  closed Riemann surface with a fixed spin structure. We consider the following coupled system of real-valued functions $\{u_k\}_{k=1}^N$ and spinors $\{\psi_k\}_{k=1}^N$ on $M$,
\begin{equation}\label{0-4}
\left\{
\begin{array}{rcl}
-\Delta u_k &=& \sum_{i=1}^{N}a_{ki}(e^{2u_i}-\frac 12e^{u_i}\left\langle \psi_i ,\psi_i
\right\rangle )-K_g,\\
\slashiii{D}\psi_k &=&\ds  -e^{u_k}\psi_k,\\
\end{array}
\right.
\end{equation}
for $k=1,2,\cdots,N$.  Here $(a_{ij})_{N\times N}$ is the Cartan matrix of $SU(N+1)$, i.e.
 \[(a_{ij})=\left(\begin{matrix}2& -1& 0 & \cdots & \cdots & 0\\ -1&2 & -1 & 0 &\cdots & 0\\ 0& -1&2&-1&\cdots & 0\\ \cdots & \cdots &\cdots &\cdots &\cdots &\cdots \\0&\cdots &\cdots &-1&2&-1\\
 0&\cdots&\cdots&0&-1&2\end{matrix}\right),
\]
$K_g$ is the Gaussian curvature of $M$ and $\slashiii{D}$ is the Dirac operator. This system couples Liouville type equations and Dirac type equations in a rather natural way. We call (\ref{0-4}) the {\em super-Toda system}.

When $\psi_k$ and $K_g$ vanish, we obtain the classical two-dimensional $SU(N+1)$ Toda system
\begin{equation*}
-\Delta u_k = \sum_{i=1}^{N}a_{ki}e^{2u_i},
\end{equation*}
for $k=1,2,\cdots,N$, which has been extensively studied in the last two decades, see e.g. \cite{JW2, JW, JLW, Lu, LWY, LL, MN, OS} etc.

When $N=1$, we have
\begin{equation}\label{equ-sl}
\left\{
\begin{array}{rcl}
-\Delta u &=& \ds\vs 2e^{2u}-e^{u}\left\langle \psi ,\psi
\right\rangle  -K_g,  \\
\slashiii{D}\psi &=&\ds  -e^{u}\psi,\\
\end{array}
\right.
\end{equation}
which is a natural  extension of the Liouville equation by a spinorial field  introduced in \cite{JWZ} and further explored in \cite{JWZZ1, JWZZ2, JZZ1, JZZ2, JZZ3}. We will also call \eqref{equ-sl} the super-Liouville equation.

In the present paper, we shall consider the case of $N=2$, where the coupled system can be written as
\begin{equation}
\left\{
\begin{array}{rcl}
-\Delta u_1 &=& \ds\vs 2e^{2u_1}-e^{2u_2}-e^{u_1}\left\langle \psi_1 ,\psi_1
\right\rangle +\frac 12 e^{u_2}\left\langle \psi_2 ,\psi_2
\right\rangle -K_g,\\
-\Delta u_2 &=& \ds\vs 2e^{2u_2}-e^{2u_1}-e^{u_2}\left\langle \psi_2 ,\psi_2
\right\rangle +\frac 12 e^{u_1}\left\langle \psi_1 ,\psi_1
\right\rangle-K_g,\\
\slashiii{D}\psi_1 &=&\ds  -e^{u_1}\psi_1,   \\
\slashiii{D}\psi_2 &=&\ds  -e^{u_2}\psi_2,   \\
\end{array}
\right.   \label{equ-1}
\end{equation}


\noindent for  the Cartan matrix of $SU(3)$
\[(a_{ij})=\left(\begin{matrix}2& -1\\ -1&2 \end{matrix}\right).
\]
The system (\ref{equ-1}) is the Euler-Lagrange system of the functional
\begin{eqnarray}\nonumber
E\left( u_1,u_2,\psi_1,\psi_2 \right) &=&\int_{M}\left \{\frac 13(
\left| \nabla u_1\right| ^2+\left| \nabla u_2\right| ^2+\nabla u_1\nabla u_2)+ K_g(u_1+u_2) \right \} dv \\
\label{cv}
& + & \int_{M}\left \{\left\langle (\slashiii{D}+e^{u_1})\psi_1
,\psi_1 \right\rangle +\left\langle (\slashiii{D}+e^{u_2})\psi_2
,\psi_2 \right\rangle-e^{2u_1}-e^{2u_2}  \right \}dvol_g.
\end{eqnarray}
It naturally couples the functional of the Toda sytem with some spinor fields  and it preserves a fundamental property of the Toda system on Riemann surfaces,  the conformal invariance. Compare this with the functional considered in \cite{O}. While, apart from differences in notation, the two functionals in the end are formally very similar, we point out that our functional involves only commuting fields. This enables us to impose inequalities like \eqref{ineq}, for instance, which would not be meaningful for anticommuting fields.

In this paper, we shall investigate the blow up behavior for limits of a sequence of solutions to the super-Toda system (\ref{equ-1}). Our  first main result is the following Brezis-Merle type concentration compactness theorem:

\begin{thm}\label{BMTH}
Let $(u_{1n},u_{2n},\psi_{1n},\psi_{2n})$ be a sequence of solutions to (\ref{equ-1}) with the following energy condition
\begin{equation}\label{ineq}
\int_{M}e^{2u_{1n}}+e^{2u_{2n}}+|\psi_{1n}|^4+|\psi_{2n}|^4<C
\end{equation}for some constant $C>0$.
  Define the blow-up sets
\[\begin{array}{rcl}
\Sigma _{u_{in}}&=&\left\{ x\in M,\text{ there is a sequence
}y_n\rightarrow x\text{ such that }u_{in}(y_n)\rightarrow +\infty
\right\}\\
\Sigma _{\psi_{in}}&=&\left\{ x\in M,\text{ there is a sequence
}y_n\rightarrow x\text{ such that }\left| \psi _{in}(y_n)\right|
\rightarrow +\infty \right\}
.\end{array}
\]
Then, we have $\Sigma_{\psi_{in}}\subset\Sigma_{u_{in}}$. Moreover,
$(u_{1n},u_{2n},\psi_{1n},\psi_{2n})$ admits a subsequence,
still denoted  by $(u_{1n},u_{2n},\psi_{1n},\psi_{2n})$, satisfying the following properties:

\begin{enumerate}
\item[a)] $\psi _{in}$ is bounded in $%
L_{loc}^{\infty} (M\backslash \Sigma _{\psi_{in}})$ .

\item[b)]  for $(u_{1n},u_{2n})$, one of the following alternatives holds:
\begin{enumerate}
\item[i)]  $(u_{1n},u_{2n})$ is bounded in $L^{\infty}(M)\times L^{\infty}(M)$ or $(u_{1n},u_{2n})\rightarrow (-\infty, -\infty)$ uniformly in  $M$.
\item[ii)]  $u_{in}$ is bounded in $L^{\infty}(M)$, but $u_{jn}\rightarrow -\infty $ uniformly in $M$ for $j\neq i$.
\item[iii)] the blow up set $\Sigma_{u_{1n}}\cup\Sigma _{u_{2n}}$ is finite and not empty. Moreover,  $u_{in}$ is bounded in $L^{\infty}_{loc}(M\backslash (\Sigma_{u_{1n}}\cup \Sigma_{u_{2n}}))$ or $u_{in}\rightarrow -\infty $ uniformly on compact subsets of $M\backslash (\Sigma_{u_{1n}}\cup \Sigma_{u_{2n}})$, for $i=1,2$.
\end{enumerate}
\end{enumerate}
\end{thm}

At a blow up point $p \in \Sigma_{u_{1n}}\cup\Sigma _{u_{2n}}$, by applying a blow-up argument, one can possibly  get an entire solution
$(u,\psi)$  of the super-Liouville system \eqref{equ-sl} on $\R^2$ with finite energy or an entire solution  $(u_1,u_2,\psi_1, \psi_2)$ of the super-Toda system \eqref{equ-1} defined on $\R^2$, i.e
\begin{equation}\label{bequ-1}
\left\{
\begin{array}{rcl}
-\Delta u_1 &=& \ds\vs 2e^{2u_1}-e^{2u_2}-e^{u_1}\left\langle \psi_1 ,\psi_1
\right\rangle +\frac 12 e^{u_2}\left\langle \psi_2 ,\psi_2
\right\rangle,  \qquad  x  \in \R^2,\\
-\Delta u_2 &=& \ds\vs 2e^{2u_2}-e^{2u_1}-e^{u_2}\left\langle \psi_2 ,\psi_2
\right\rangle +\frac 12 e^{u_1}\left\langle \psi_1 ,\psi_1
\right\rangle, \qquad  x  \in \R^2, \\
\slashiii{D}\psi_1 &=&\ds  -e^{u_1}\psi_1, \qquad  x \in \R^2, \\
\slashiii{D}\psi_2 &=&\ds  -e^{u_2}\psi_2, \qquad  x \in \R^2, \\
\end{array}
\right.
\end{equation}
with the finite energy condition
\begin{equation}
\int_{\R^2}(e^{2u_1}+e^{2u_2}+|\psi_1|^4+|\psi_2|^4)<\infty. \label{bequ-2}
\end{equation}
Recall that in the super-Liouville case, it was shown in \cite{JWZ} that the following integral quantity is quantized
\begin{equation} \label{entire-energy-sl}
\alpha= \int_{\R^2}e^{2u}-\frac 12 e^{u}|\psi|^2=2 \pi.
\end{equation}
In the super-Toda case, naturally, we shall define the following two integral quantities
\begin{equation} \label{entire-energy-st}
\alpha_i : =\int_{\R^2}e^{2u_i}-\frac 12 e^{u_i}|\psi_i|^2, \quad i=1, 2.
\end{equation}

\

 By applying the Kelvin transformation and  potential analysis, we have the following asymptotic behavior of an entire solution for the super-Toda system. In particular, we show that the two quantities defined in \eqref{entire-energy-st} are both larger than $2\pi$.

\begin{thm}\label{asyth}
Let $(u_1,u_2,\psi_1,\psi_2)$ be a solution of (\ref{bequ-1}) and (\ref{bequ-2}). Then we have
\begin{equation}\label{ayu}
u_i(x)=-\frac {\beta_i}{2\pi} \ln{|x|}+C+O(|x|^{-1}),
\end{equation}
\begin{equation}\label{asy-psi-4}
|\psi_i (x)|\leq C|x|^{-\frac 12-\delta_0},
\end{equation}
for $|x|$ large and for some small positive constant $\delta_0$, where $\beta_1 =2\alpha_1-\alpha_2$, $\beta_2 =2\alpha_2-\alpha_1$ and $\alpha_i, i=1,2$ are defined as in \eqref{entire-energy-st}. Moreover, there hold
\begin{equation}\label{alpha}
\alpha_1^2+\alpha_2^2-\alpha_1\alpha_2=2\pi(\alpha_1+\alpha_2),
\end{equation}
\begin{equation}\label{beta}
\beta_1^2+\beta_2^2+\beta_1\beta_2=6\pi(\beta_1+\beta_2),
\end{equation}
with $\beta_i>2\pi$  and $\alpha_i>2\pi$, for $i=1,2$.
\end{thm}

We remark that the identities in \eqref{alpha}, \eqref{beta} were shown with the help of a holomorphic quadratic differential defined for a solution of the super-Toda system (see Proposition \ref{pro-holo}).

By further exploring the asymptotic analysis of the blow-up sequence  $(u_{1n},u_{2n},\psi_{1n},\psi_{2n})$, we shall see that, in addition to the above mentioned two types of bubbling solutions defined on $\R^2$, there possibly exist two more types of bubbling solutions, namely, solutions of \eqref{equ-1} or \eqref{equ-sl} defined on $\R^2\setminus \{0\}$ and with finite energy. In general, a global or a local singularity for such solutions may not be removable and hence these solutions may not be conformally extended to solutions on $\S^2$. In the case of the super-Liouville equation, we can show that all the singularities arising from the bubbling solutions are removable (Theorem 1.2 in \cite{JWZZ1}), by exploring the fact that the corresponding holomorphic quadratic differential is locally $L^1$ integrable near the singularities (Proposition 2.6 in \cite{JWZZ1}). However, in the case of the super-Toda system where there are two scalar fields and two spinor fields, the approach in \cite{JWZZ1}  fails and hence we need to develop a new method, without relying on some singularity removability result.

In this paper, we develop  such a new method. In fact, in our new approach, we are able to show that all these four types of  bubbling solutions on $\R^2$ or on $\R^2\setminus \{0\}$ have a new kind of energy gap property for the spinor equation (see Lemma \ref{spinor-energygap}).   Consequently, we have the following energy identities for the spinor parts of a blow-up sequence of solutions to the super-Toda system \eqref{equ-1} with uniformly bounded energy:

\begin{thm}\label{engy-indt}
Let $(u_{1n},u_{2n},\psi_{1n},\psi_{2n})$ be a sequence of solutions to (\ref{equ-1}) with uniformly bounded energy
\begin{equation}\label{equ-2n}
\int_{M}e^{2u_{1n}}+e^{2u_{2n}}+|\psi_{1n}|^4+|\psi_{2n}|^4 <C<\infty
\end{equation}
Then, there are at most finitely many bubbling solutions of the following four types:
 \begin{itemize}
 \item[(I.1-I.2)] nontrivial solutions of (\ref{equ-1}) on $\R^2$ or on $\R^2\setminus{\{0\}}$ with finite energy:  \\
 $(u_1^{i,k},u_2^{i,k}, \psi_1^{i,k}, \psi_2^{i,k})$, $i=1,2,\cdots , l; k=1,2,\cdots, L_i$ satisfying
 \begin{equation}\label{bubbleproperty-st}
\epsilon_1 \leq \int_{\R^2}e^{2u_1^{i,k}}+e^{2u_2^{i,k}}, \quad    \int_{\R^2} |\psi_1^{i,k}|^4+|\psi_2^{i,k}|^4< \infty,   \quad \int_{\R^2}e^{2u_j^{i,k}}-\frac 12 e^{u_j^{i,k}}|\psi_j^{i,k}|^2 > 0, \quad  j=1,2.
\end{equation}

 \item[(II.1-II.2)]  nontrivial solutions of (\ref{equ-sl}) on $\R^2$ or on $\R^2\setminus{\{0\}}$ with finite energy: \\
 $(u_j^{i,k},\phi_j^{i,k})$, $i=1,2,\cdots , m; k=1,2,\cdots, M_i$, $j=1,2$
satisfying
 \begin{equation}\label{bubbleproperty-sl}
\epsilon_1 \leq \int_{\R^2}e^{2u_j^{i,k}}, \quad   \int_{\R^2} |\phi_j^{i,k}|^4 < \infty, \quad
\int_{\R^2}e^{2u_j^{i,k}}-\frac 12 e^{u_j^{i,k}}|\psi_j^{i,k}|^2 > 0.
\end{equation}

\end{itemize}
where $\epsilon_1>0$ is some universal constant as in Lemma \ref{spinor-energygap}, such that, after selection of a subsequence, $(\psi_{1n},\psi_{2n})$
converges in $C_{loc}^{2}$ to some $(\psi_1,\psi_2)$ on $M\backslash (\Sigma_{1n}\cup \Sigma_{2n})$
and the following energy identities hold: \begin{equation*}
\lim_{n\rightarrow
\infty}\int_{M}|\psi_{jn}|^4=\int_{M}|\psi_j|^4+\sum_{i=1}^{l}
\sum_{k=1}^{L_i}\int_{\R^2}|\psi_j^{i,k}|^4+\sum_{i=1}^{m}
\sum_{k=1}^{M_i}\int_{\R^2}|\phi_j^{i,k}|^4,     \qquad j=1, 2.
\end{equation*}
\end{thm}

\
\

For each blow-up point $p\in \Sigma_{u_{1n}}\cup \Sigma_{u_{2n}}$,  we define the blow-up values $(m_1(p),m_2(p))$ at $p$ as follows:
\begin{equation} \label{blow-up-value-st}
m_i(p)=\lim_{r\rightarrow 0}\lim_{n\rightarrow \infty}\int_{B_r(p)}e^{2u_{in}}-\frac 12e^{u_{in}} | \psi_{in}|^2, \quad i=1, 2.
\end{equation}

As an application of the energy identities for the spinor parts of a blow-up sequence of solutions to the super-Toda system, we can improve the concentration compactness result in Theorem \ref{BMTH}.

\begin{thm}\label{Pblowupb}
If $\Sigma_{u_{1n}}\cup \Sigma_{u_{2n}}\neq \emptyset$ in Theorem \ref{BMTH}, then at least one of  $u_{1n}$ and $u_{2n}$ tends to $-\infty$ uniformly in any compact subset of $M\backslash   (\Sigma_{u_{1n}}\cup  \Sigma_{u_{2n}})$. Moreover, let $(m_1(p),m_2(p))$ be the blow-up values  at a blow-up point $p$, then the following three alternatives hold:
\begin{itemize}
\item[i)]  $m_1(p) > 2 \pi, m_2(p) > 2 \pi$;

\item[ii)]  $m_1(p) \geq 2\pi,  m_2(p) \geq 0$;

\item[iii)]  $m_1(p) \geq 0,  m_2(p) \geq 2\pi$.
\end{itemize}

\end{thm}

 In the proof of Theorem \ref{Pblowupb}, we need  two crucial  facts. Firstly, the integral quantities defined in \eqref{entire-energy-sl} and \eqref{entire-energy-st} for finite energy solutions of \eqref{equ-1} and \eqref{equ-sl}  on $\R^2$ are all larger or equal to $2\pi$. Secondly, for finite energy solutions of \eqref{equ-1} and \eqref{equ-sl}  on $\R^2\setminus \{0\}$, one can define some integral quantities similarly to those in \eqref{entire-energy-sl} and \eqref{entire-energy-st} and it turns out that these quantities are all positive (see Theorem \ref{nonnegative-st} and  Theorem \ref{nonnegative-sl}).

\
\

 \section{Spinors and Geometric Properties of the Super-Toda System }

\

In this section, we will state some geometric properties of the super-Toda system. Let us start with some background about spin structures and spinors.
Let $(M,g)$ be a  closed Riemann surface and let $P_{SO(2)}\to M$ be its
oriented orthogonal frame bundle. A $Spin$-structure is a lift of
the structure group SO(2) to $Spin(2)$.

Let $\Sigma ^+M:=P_{Spin(2)}\times_{\rho}{\mathbb C}$ be a complex
line bundle over $M$ associated to $ P_{Spin(2)}$ and to the
standard representation $\rho:\S^1\to U(1)$. This is the bundle of
positive half-spinors. Its complex conjugate $\Sigma^-M:=
\overline {\Sigma^+M}$ is called the bundle of negative
half-spinors. The spinor bundle is $\Sigma M:=\Sigma ^+M\oplus
\Sigma ^-M.$
There exists a Clifford multiplication
\begin{eqnarray*}
TX\times_{\C} \Sigma^+M &\to &  \Sigma^-M \cr TX\times_{\C}
\Sigma^-M &\to &  \Sigma^+M \end{eqnarray*} denoted by $v\otimes
\psi \to v \cdot \psi$, which satisfies the  Clifford relations
\[ v\cdot w \cdot\psi + w\cdot v\cdot \psi =-2g(v,w) \psi,\]
for all $v,w \in TM$ and $\psi \in \Gamma(\Sigma M)$.

On the spinor bundle $\Sigma M$, there is a natural Hermitian
metric $ \la \cdot,\cdot \ra $. We also denote $|\cdot|^2= \la \cdot,\cdot \ra $.  Let $\nabla$ be
the Levi-Civita connection on $M$ with respect to $g$. Likewise, $\nabla$ induces a
connection (also denoted by $\nabla$) on  $\Sigma M$ that is compatible
with the Hermitian metric.

The Dirac operator $\slashiii{D}$ is defined by $\slashiii{D}\psi
:=\sum_{\alpha =1}^2e_{_\alpha }\cdot \nabla _{e_\alpha }\psi ,$
where $\left\{ e_1,e_2\right\} $ is a local orthogonal frame on $TM$.
For more details about spin geometry and Dirac operator, we
refer to \cite{LM}.

\
\
Next we will show that the fundamental conformal invariance of the action functional for the Toda-system is preserved in the super-Toda case.
\begin{prop}For any conformal diffeomorphism $\varphi :M\rightarrow M$  with $\varphi^*(g)=e^{2w} g$,  set
\beq\label{n3.1}\ba{rcl}
\ds\vs \widetilde{u_i} &=&\ds u_i\circ \varphi -w,  \\
\ds \widetilde{\psi_i }&=&\ds e^{-\frac 12 w}\psi_i \circ \varphi, \ea
\eeq
for $i=1,2$. Then if $(u_1,u_2,\psi_1,\psi_2)$ is a solution of
(\ref{equ-1}), so is $(\tilde u_1, \tilde u_2, \tilde \psi_1, \tilde \psi_2)$. That is,  the super-Toda system is conformally invariant.
\end{prop}
\begin{proof} The proof is similar to the case of the super-Liouville equation in Proposition 3.1 in \cite{JWZ}. Here we provide a sketch of the proof.

Let $\tilde g=\varphi^*g$ and let $\slashiii{\tilde{D}}$ be the Dirac operator with respect to the new metric $\tilde{g}$.
Notice that the relation between the two curvatures $R_g$ and $R_{\tilde{g}}$ is
$$
\Delta_{\tilde{g}}w+K_{\tilde{g}}=K_ge^{-2w},
$$
and the relation between the two Dirac operators $\slashiii{D}$ and
$\widetilde{\slashiii{D}}$ is
\[
\widetilde{\slashiii{D}}\widetilde{\psi }=e ^{-\frac 32w}\slashiii{D}(e ^{\frac 12w}
\widetilde{\psi })=e^{-\frac 32w}\slashiii{D}\psi.
\]
The conclusion can then be deduced by a direct computation.
\end{proof}

Analogously to the case of the Toda system \cite{JLW} and  of the  super-Liouville equation \cite{JWZ}, we can define a holomorphic quadratic differential associated to a solution of the super-Toda system.
\begin{prop}\label{pro-holo}
Let $(u_1,u_2,\psi_1,\psi_2)$ be a solution of (\ref{equ-1}) on  $M$. Let $z=x+iy$ be a local isothermal parameter with $g=\rho |dz|^2$. Then the quadratic differential
\[
T(z)dz^2=\left\{\sum_{j,k=1}^2 a^{jk}((u_k)_{zz}- (u_j)_z (u_k)_z-\frac 18\langle
\psi_k,dz\cdot\partial_{\bar z}\psi_k\rangle -\frac 18\langle
d\bar{z}\cdot\partial_z\psi_k,\psi_k\rangle ) \right \}dz^2
\]
is holomorphic if $M$ has constant scalar curvature. Here $dz=dx+idy$, $d\bar z =dx-idy$ and $(a^{ij})$ is the inverse of the Cartan matrix $(a_{ij})$.
\end{prop}

\begin{proof}Such a result for the case of super-Liouville equation was presented in \cite{JWZ}. Here we follow the same idea to prove this proposition.
Set
$$
T_1(z)=\sum_{j,k=1}^2 a^{jk}((u_k)_{zz}- (u_j)_z (u_k)_z),
$$
and
$$
T_2(z)=\sum_{j,k=1}^2 a^{jk}(\langle
\psi_k,dz\cdot\partial_{\bar z}\psi_k\rangle +\langle
d\bar{z}\cdot\partial_z\psi_k,\psi_k\rangle).
$$
Noting that
$$ (a^{ij})=\left(\begin{matrix}\frac 23 & \frac 13 \\ \frac 13 & \frac 23 \end{matrix}\right),$$
we have
$$
T_2(z)=(\langle
\psi_1,dz\cdot\partial_{\bar z}\psi_1\rangle +\langle
d\bar{z}\cdot\partial_z\psi_1,\psi_1\rangle)+(\langle
\psi_2,dz\cdot\partial_{\bar z}\psi_2\rangle +\langle
d\bar{z}\cdot\partial_z\psi_2,\psi_2\rangle).
$$
By using the symmmetry of $(a^{ij})$ and the Ricci curvature formula, we obtain
\begin{eqnarray*}
\partial_{\bar{z}}T_1(z)&=&\sum_{j,k=1}^2 a^{jk}((u_k)_{zz\bar{z}}- 2(u_j)_z (u_k)_{z\bar{z}})\\
&=&\frac 14\sum_{j,k=1}^2 a^{jk}((\triangle u_k)_{z}+2K_g(u_{k})_z-2(u_j)_z \triangle u_k)\\
&=&\frac 18\sum_{j=1}^2( -e^{u_j}\left| \psi_j \right| ^2 (u_j)_z+
e^{u_j}\partial_z\left| \psi_j \right|^2)+\frac 12(K_g)_z.
\end{eqnarray*}

Next we let $(e_1,e_2)$ be a local orthonormal basis on $M$ such that $\nabla_{e_\alpha}e_\beta=0$ . Recall the Clifford multiplication
$$
e_i\cdot e_j\cdot \psi+e_j\cdot e_i\cdot \psi=-2\delta_{ij}\psi
$$
and
$$
\langle\psi, \phi\rangle=\langle e_i\cdot\psi, e_i\cdot\phi\rangle
$$for $i,j=1,2$ and for any spinor $\psi,\phi$.
Therefore if we write
$$2\text{ Re }\langle \psi,\phi\rangle=
\langle \psi,\phi\rangle+\langle \phi,
\psi\rangle,
$$ we have $\text{ Re }\langle e_i\cdot \psi,\psi\rangle=0$, and $\text{Re}\left\langle \psi
,e_i\cdot \nabla _{e_j }\psi \right\rangle -\text{Re}\left\langle \psi
,e_j\cdot \nabla _{e_i }\psi \right\rangle=\text{Re}\left\langle e_i\cdot e_j\cdot\psi
,\slashiii{D}\psi \right\rangle$ for $i\neq j$. Therefore, for a solution of $(u_k,\psi_k)$,   $\text{Re}\left\langle \psi_k
,e_i\cdot \nabla _{e_j }\psi_k \right\rangle$ is symmetric with respect to the indices $i,j$.
Now we can compute that
\begin{eqnarray*}
&& \partial_{\bar{z}} (\langle
\psi_i,dz\cdot\partial_{\bar z}\psi_i\rangle +\langle
d\bar{z}\cdot\partial_z\psi_i,\psi_i\rangle)\\
&=&\frac 12\partial_{\bar{z}}(\langle(e_1-ie_2)\cdot(\nabla_{e_1}\psi_i-i\nabla_{e_2}\psi_i),\psi_i\rangle+
\langle\psi_i, (e_1+ie_2)\cdot(\nabla_{e_1}\psi_i+i\nabla_{e_2}\psi_i)\rangle)\\
&=&
\partial_{\bar{z}}(\text{Re}\langle\psi_i,e_1\cdot\nabla_{e_1}\psi_i\rangle
-2i\text{Re}\langle\psi_i,e_1\cdot\nabla_{e_2}\psi_i\rangle-\text{Re}\langle\psi_i,e_2\cdot\nabla_{e_2}\psi_i\rangle)\\
&=& \frac 12 (-2i\text{Re}\langle
\nabla _{e_1}\psi_i,e_1\cdot\nabla_{e_2}\psi_i\rangle-\text{Re}\langle
\nabla
_{e_1}\psi_i,e_2\cdot\nabla_{e_2}\psi_i\rangle)\\
&& +\frac 12(i\text{Re}\langle \nabla
_{e_2}\psi_i,e_1\cdot\nabla_{e_1}\psi_i\rangle+2\text{Re}\langle
\nabla
_{e_2}\psi_i,e_2\cdot\nabla_{e_1}\psi_i\rangle)\\
&& +\frac 12 (\text{Re}\langle
\psi_i,e_1\cdot\nabla_{e_1}\nabla_{e_1}\psi_i\rangle-2i\text{Re}\langle
\psi_i,e_1\cdot\nabla_{e_1}\nabla_{e_2}\psi_i\rangle-\text{Re}\langle
\psi_i,e_2\cdot\nabla_{e_1}\nabla_{e_2}\psi_i\rangle)\\
&&+\frac 12 (i\text{Re}\langle
\psi_i,e_1\cdot\nabla_{e_2}\nabla_{e_1}\psi_i\rangle+2\text{Re}\langle
\psi_i,e_2\cdot\nabla_{e_2}\nabla_{e_1}\psi_i\rangle-i\text{Re}\langle
\psi_i,e_2\cdot\nabla_{e_2}\nabla_{e_2}\psi_i\rangle).\\
\end{eqnarray*}
By using the definition of  the
curvature operator $R^{\Sigma M}$  of the connection $\nabla $ on
the spinor bundle $\Sigma M$, that is
\[
\nabla _{e_{\alpha} }\nabla _{e_{\beta }}\psi -\nabla _{e_{\beta}
}\nabla _{e_{\alpha }}\psi =R^{\Sigma M}(e_{\alpha },e_{\beta}
)\psi ,
\]
and a formula for this curvature operator (see e.g. \cite{Jo})
\[
\sum_{\alpha =1}^2e_{\alpha} \cdot R^{\Sigma M}(e_{\alpha} ,X)\psi
=\frac 12Ric(X)\cdot \psi ,\text{ \qquad for }\forall X\in \Gamma
(TM)
\]
we have
$$
\text{Re}\langle \psi,e_2\cdot R^{\Sigma M
}(e_1,e_2)\psi\rangle=\text{Re}\langle \psi,\frac 12
Ric(e_1)\cdot\psi\rangle=0,
$$
and
$$
\text{Re}\langle \psi,e_1\cdot R^{\Sigma M
}(e_1,e_2)\psi\rangle=\text{Re}\langle \psi,\frac 12
Ric(e_2)\cdot\psi\rangle=0.
$$
Hence we obtain that
\begin{eqnarray*}
&& \partial_{\bar{z}} (\langle
\psi_i,dz\cdot\partial_{\bar z}\psi_i\rangle +\langle
d\bar{z}\cdot\partial_z\psi_i,\psi_i\rangle)\\
&=& \frac 12(-3\text{Re}\langle \nabla
_{e_1}\psi_i,e_2\cdot\nabla_{e_2}\psi_i\rangle+3i\text{Re}\langle
\nabla
_{e_2}\psi_i,e_1\cdot\nabla_{e_1}\psi_i\rangle)\\
&& +\frac 12(\text{Re}\langle
\psi_i,\nabla_{e_1}(\slashiii{D}\psi_i)\rangle-i\text{Re}\langle
\psi_i,\nabla_{e_2}(\slashiii{D}\psi_i)\rangle)\\
& =& \frac 12(-3\text{Re}\langle \nabla
_{e_1}\psi_i,\slashiii{D}\psi_i-e_1\cdot\nabla_{e_1}\psi_i\rangle+3i\text{Re}\langle
\nabla
_{e_2}\psi_i,\slashiii{D}\psi_i-e_2\cdot\nabla_{e_2}\psi_i\rangle)\\
&& +\frac 12(\text{Re}\langle
\psi_i,\nabla_{e_1}(\slashiii{D}\psi_i)\rangle-i\text{Re}\langle
\psi_i,\nabla_{e_2}(\slashiii{D}\psi_i)\rangle)\\
&=& e^{u_i}\partial_z|\psi_i|^2- e^{u_i}|\psi_i|^2(u_i)_z.
\end{eqnarray*}
Consequently we have
$$
\partial_{\bar{z}}T(z)=\partial_{
\bar{z}}T_1(z) -\frac 18\partial_{\bar{z}}T_2(z)=\frac{1}{2}(K_g)_z.
$$
That means $T(z)$ is holomorphic if $K_g$ is constant.  This completes the proof.
\end{proof}

\

\section{Asymptotic behavior of entire solutions}

\

In this section, we will analyze the asymptotic behavior near  infinity of an entire solution $(u_1,u_2, \psi_1,\psi_2)$ of \eqref{bequ-1} on $\R^2$ with finite energy.

Consider the  Kelvin transformation
\begin{eqnarray*}
&& v_i(x)=u_i(\frac {x}{|x|^2})-2\ln |x|,\\
&& \phi_i (x)=|x|^{-1} \psi_i (\frac{x}{|x|^2}),
\end{eqnarray*}
for $i=1,2$.
By conformal invariance, $(v_1,v_2,\phi_1,\phi_2)$  satisfies
\begin{equation}\label{sequ-1}
\left\{
\begin{array}{rcll}
-\Delta v_1 &=& 2e^{2v_1}-e^{2v_2}-e^{v_1}\left\langle \phi_1 ,\phi_1
\right\rangle+\frac 12 e^{v_2}\left\langle \phi_2 ,\phi_2
\right\rangle,&\qquad x\in \R^2\backslash\{0\}\\
-\Delta v_2 &=& 2e^{2v_2}-e^{2v_1}-e^{v_2}\left\langle \phi_2 ,\phi_2
\right\rangle+\frac 12 e^{v_1}\left\langle \phi_1 ,\phi_1
\right\rangle,&\qquad x\in \R^2\backslash\{0\}\\
\slashiii{D}\phi_1 &=&\ds  -e^{v_1}\phi_1, &\qquad x\in
\R^2\backslash\{0\}\\
\slashiii{D}\phi_2 &=&\ds  -e^{v_2}\phi_2, &\qquad x\in
\R^2\backslash\{0\}.
\end{array}
\right.
\end{equation}

Therefore, from Lemma 6.2 in \cite{JWZ}, we have the following asymptotic behavior near the origin for the new spinors $\phi_i$:

\begin{lm}\label{asy-psi}
There is a small positive constant  $\varepsilon_0 $, such that,  if $(v_1,v_2,\phi_1,\phi_2)$ is a smooth
solution to (\ref{sequ-1}) in $B_1\backslash\{0\}$ with energy
$\int_{|x|\leq 1}e^{v_i}dx<\varepsilon_0$ and $\int_{|x|\leq
1}|\phi_i|^4dx<C$ for $i=1,2$, then for any $x\in B_{\frac {1}{2}}$ we have
\begin{equation}\label{asy-psi-1}
|\phi_i(x)||x|^{\frac 12}+|\nabla\phi_i(x)||x|^{\frac 32}\leq
C(\int_{B_{2|x|}}|\phi_i|^4dx)^{\frac 14}, \quad i=1, 2.
\end{equation}
Furthermore, if we assume that $e^{2v_i}=O(\frac
{1}{|x|^{2-\varepsilon}})$ for $i=1,2$, then,  for any $x\in B_{\frac 12}$, we
have
\begin{equation}\label{asy-psi-2}
|\phi_i(x)||x|^{\frac 12}+|\nabla\phi_i(x)||x|^{\frac 32}\leq
C|x|^{\frac {1}{4C}}(\int_{B_1}|\phi_i|^4dx)^{\frac 14}, \quad i=1, 2.
\end{equation}
for some positive constant $C$. Here $\varepsilon$ is any
sufficiently small positive number.
\end{lm}

By using the Kelvin transformation again, we obtain the asymptotic behavior near  infinity for the original spinors $\psi_i$:
\begin{equation}\label{asy-psi-3}
|\psi_i (x)|\leq C|x|^{-\frac 12-\delta_0}\qquad \text{for}\quad |x|
\quad \text{near}\quad \infty
\end{equation}
for some positive number $\delta_0$,  provided that $e^{2u_i}=O(\frac
{1}{|x|^{2+\varepsilon}})$ near $\infty$.

\

The next lemma gives some regularity property for an entire solution.

\begin{lm}\label{blm}
Let $(u_1,u_2, \psi_1,\psi_2)$ be a solution of (\ref{bequ-1}) and (\ref{bequ-2}) with
$u_i\in H^{1,2}_{loc}(\R^2)$ and $\psi_i\in H^{1,\frac
43}_{loc}(\R^2)$ for $i=1,2$. Then $u_i^{+}\in L^{\infty}(\R^2)$ for $i=1,2$.
\end{lm}

\begin{proof} We have the Brezis-Merle inequality \cite{BM}:
\begin{equation}\label{BMInq}
\int_{\Omega}\text{exp}\{\frac{(4\pi-\delta)|u(x)|}{||f||_{L^1(\Omega)}}\}dx\leq \frac {4\pi^2}{\delta}(\text { diam } \Omega)^2
\end{equation}
which holds for every $\delta \in (0,4\pi)$ and for every solution $u$ of
\begin{equation*}
\left\{
\begin{array}{rcll}
-\triangle u&=& f(x)\qquad &\text{ in } \Omega\\
u&=& 0\qquad &\text{ on } \partial \Omega
\end{array}
\right.
\end{equation*}
with $f\in L^1(\Omega)$ and $\Omega$  a bounded domain in $\R^2$.

To prove this lemma, by conformal invariance, it suffices to show that , for any $x_0 \in \R^{2}$, $u^+_i$  is bounded  on $B_{1}(x_0)$ and the bounds are independent of the point $x_0$. In the rest of the proof, we shall use $C$ to denote various constants independent of $x_0$.

We consider $u_1$ first. Denote $$f_1=2e^{2u_1}-e^{2u_2}-e^{u_1}\left\langle \psi_1 ,\psi_1
\right\rangle +\frac 12 e^{u_2}\left\langle \psi_2 ,\psi_2
\right\rangle.$$
It follows from the energy condition (\ref{bequ-2}) that $f_1 \in L^1(B_3(x_0))$. We write $f_1=f_{11}+f_{12}$ with
$||f_{11}||_{L^1(B_3(x_0))}\leq \pi$ and $f_{12}\in
L^{\infty}(B_3(x_0))$.
Define $u_{11}$, $u_{12}$ by
\begin{equation*}
\left\{
\begin{array}{rcll}
-\triangle u_{11} &=& f_{11}  &\qquad \text{in} \quad B_{3}(x_0)\\
u_{11}&=& 0, &\qquad\text{on}\quad \partial B_{3}(x_0).
\end{array}
\right. \end{equation*}
and
\begin{equation*}
\left\{
\begin{array}{rcll}
-\triangle u_{12} &=& f_{12}  &\qquad \text{in} \quad \quad B_{3}(x_0)\\
u_{12} &=& 0, &\qquad\text{on}\quad \partial B_{3}(x_0)
\end{array}
\right.
\end{equation*}
By using the inequality (\ref{BMInq}) we have
$$
\int_{B_3(x_0)}\exp(2|u_{11}|)dx\leq C,
$$
hence $||u_{11}||_{L^q(B_3(x_0))}\leq C$ for any $q\geq 1$. We also have
$$
||u_{12}||_{L^{\infty}(B_3(x_0))}\leq C.
$$
\noindent Now let
$u_{13}=u_1-u_{11}-u_{12}$. Then  $u_{13}$ is a harmonic function in $B_{3}(x_0)$. The mean value theorem for harmonic functions implies that
$$
||u^+_{13}||_{L^{\infty}(B_2(x_0))}\leq C ||u^+_{13}||_{L^{1}(B_3(x_0))}.
$$
Since $u^+_{13}\leq u^+_{1} +|u_{11}|+|u_{12}|$, and $
\int_{ B_{3}(x_0)}{u_1}^+dx\leq
\int_{ B_{3}(x_0)}e^{u_1}dx<\infty,
$
we get
$$
||u^+_{13}||_{L^{\infty}(B_{2}(x_0))}\leq C.
$$
For $u_2$, by a similar argument as for $u_1$, we can write $u_2=u_{21}+u_{22}+u_{23}$ with
$$
||u_{21}||_{L^q(B_3(x_0))}\leq C,\qquad ||u_{22}||_{L^{\infty}(B_3(x_0))}\leq C, \qquad ||u^+_{23}||_{L^{\infty}(B_{2}(x_0))}\leq C.
$$
It is clear that $$||\psi_i||_{L^{\infty}(B_2(x_0))}\leq C$$ for $i=1,2$.  If we rewrite $f_1$ as
$$
f_1=2e^{2u_{12}+2u_{13}}e^{2u_{11}}-e^{2u_{22}+2u_{23}}e^{2u_{21}}-e^{u_{12}+u_{13}}e^{u_{11}}\left\langle \psi_1 ,\psi_1
\right\rangle +\frac 12 e^{u_{22}+u_{23}} e^{u_{21}}\left\langle \psi_2 ,\psi_2
\right\rangle
$$
then we know that $||f_1||_{L^q(B_{2}(x_0))}\leq C$ for any $q>1$. Then standard elliptic estimates imply that
$$
||u_{1}^+||_{L^{\infty}(B_{1}(x_0))}\leq C||u_{1}^+||_{L^{1}(B_{2}(x_0))}+C||f_1||_{L^q(B_{2}(x_0))}\leq C.
$$
Similarly, we also have
$$
||u_{2}^+||_{L^{\infty}(B_{1}(x_0))}\leq C.
$$
\end{proof}

 Next we will analyze the asymptotic behavior near infinity of an entire solution. We shall apply  standard potential analysis as in \cite{CL1} and \cite{JWZ}.

\

\noindent{\bf Proof of Theorem \ref{asyth}:}
First, for $k=1,2$  we define
$$
w_k(x)=-\frac{1}{2\pi}\int_{\R^2}(\ln{|x-y|}-\ln{(|y|+1)})( \sum_{i=1}^{2}a_{ki}(e^{2u_i}-\frac 12e^{u_i}\left\langle \psi_i ,\psi_i
\right\rangle ))dy,
$$
 Since $\sum_{i=1}^{2}a_{ki}(e^{2u_i}-\frac 12e^{u_i}\left\langle \psi_i ,\psi_i
\right\rangle )$ is $L^1$ integrable in $\R^2$, we get
$$
\frac {w_k(x)}{\ln{|x|}}\rightarrow -\frac
{1}{2\pi}\int_{\R^2} \sum_{i=1}^{2}a_{ki}(e^{2u_i}-\frac 12e^{u_i}\left\langle \psi_i ,\psi_i
\right\rangle )dx=-\frac{\beta_k}{2\pi}
$$ as $ |x|\rightarrow +\infty $ uniformly.
Furthermore,  we have $-\triangle w_k(x)=\sum_{i=1}^{2}a_{ki}(e^{2u_i}-\frac 12e^{u_i}\left\langle \psi_i ,\psi_i
\right\rangle )$ on $\R^2$. Therefore, if we define
$v_k=u_k(x)-w_k(x)$, then $\triangle v_k(x)=0$ on $\R^2$.
Since $u^{+}_{k}\in L^{\infty}(\R^2)$ by Lemma \ref{blm}, we get that
$$
v_k(x)\leq C_1+C_2\ln{|x|},
$$
for $|x|$ sufficiently large, with $C_1$ and $C_2$ being two positive constants.
Therefore, by Liouville's theorem on harmonic functions, $v_k(x)$
has to be constant and hence we get
$$
\frac {u_k(x)}{\ln{|x|}}\rightarrow -\frac {\beta_k}{2\pi}\qquad
\text{as} \quad |x|\rightarrow +\infty, \quad \text{uniformly}.
$$
Since $\int_{\R^2}e^{2u_k}dx<+\infty$, the above result implies
$$
\beta_k \geq 2\pi.
$$

Next,  we will show that $\beta_k >2\pi$ for $k=1,2$. Suppose this is not true, we assume without loss of generality that $\beta_1=2\pi$.

\

\noindent {\bf Claim.} $\beta_2=2\pi$.

\

We shall prove the claim by contradiction.  We assume that $\beta_2>2\pi$. Let $(v_k,\phi_k)$ be the Kelvin transformation of
$(u_k,\psi_k)$. Then $(v_k,\phi_k)$ satisfy (\ref{sequ-1}) in
$B_1\backslash\{0\}$. Denote
$$
f_1=2e^{2v_1}-e^{2v_2}-e^{v_1}\left\langle \phi_1 ,\phi_1
\right\rangle+\frac 12 e^{v_2}\left\langle \phi_2 ,\phi_2
\right\rangle,
$$
then we have
$$
-\triangle v_1=f_1(x) \qquad \text{in}\quad B_\delta\backslash\{0\}.
$$
From the asymptotic estimate (\ref{asy-psi-1}) of $\phi_i$, we know that
$f_1(x)>0$ in a small punctured disk $B_\delta\backslash\{0\}$. Set
$$
h(x)=-\frac{1}{2\pi}\int_{B_\delta}\log{|x-y|f_1(y)dy}
$$
and let $g(x)=v_1(x)-h(x)$. It is easy to see that $\triangle h=-f_1$ and
$\triangle g=0$.

On the other hand, we can verify that
$$
\lim_{|x|\rightarrow 0}\frac {h}{-\log{|x|}}=0,
$$
and it follows that
$$
\lim_{|x|\rightarrow 0}\frac
{g(x)}{-\log{|x|}}=\lim_{|x|\rightarrow 0}\frac
{v_1(x)-h(x)}{-\log{|x|}}=\lim_{|x|\rightarrow 0}\frac {u_1(\frac
{x}{|x|^2})-2\log{|x|}}{-\log{|x|}}=1.
$$
Since $g(x)$ is harmonic in $B_\delta\backslash\{0\}$, we get
$g(x)=-\log{|x|}+g_0(x)$ with a smooth harmonic function $g_0$ in
$B_\delta$. By definition, we have $h(x)>0$. Thus, we obtain
$$
\int_{B_\delta}e^{2v_1}dx=\int_{B_\delta}e^{2g+2h}dx\geq \int_{B_\delta}\frac
{1}{|x|^2}e^{2g_0}dx=+\infty,
$$
which is a contradiction with $\int_{\R^2}e^{2v_1}dx<\infty$. Hence
the claim holds.

\

Next let us show that $\beta_1=\beta_2=2\pi$ is impossible. We assume by contradiction that $\beta_k=2\pi$ for $k=1, 2$. We consider $w=\frac 23 v_1+\frac 13 v_2$. It is clear that
$$
-\triangle w=e^{2v_1}-\frac 12 e^{v_1}|\phi_1|^2
\qquad \text { in } B_\delta \backslash \{0\},$$
with
$$
\frac{w(x)}{\ln{|x|}} \rightarrow -1 \qquad
\text{as} \quad |x|\rightarrow 0 \quad \text{uniformly},
$$
and
$$
\int_{\R^2}e^{2w}dx<\infty.
$$ Since one can check that $e^{2v_1}-\frac 12 e^{v_1}|\phi_1|^2>0$ in a small punctured disk $B_\delta \backslash \{0\}$, we  can get a contradiction by applying a similar argument as in the proof of the above claim.

Thus we have shown that $\beta_k>2\pi$ for $k=1, 2$.

\

At this point, we can prove the asymptotic estimate (\ref{ayu}). Since  $\beta_i>2\pi$ for $i=1,2$, we can improve the estimate for $e^{2u_i}$ to the following
\begin{equation*}
e^{2u_i}\leq C|x|^{-2-\varepsilon}\qquad \text{for}\quad |x| \quad
\text{near}\quad \infty
\end{equation*} for $i=1,2 $. Then by using  potential analysis we  get
$$
-\frac{\beta_i}{2\pi}\ln{|x|}-C\leq u_i(x)\leq
-\frac{\beta_i}{2\pi}\ln{|x|}+C
$$
for some constant $C>0$, see \cite{CL2}.

By applying a similar argument as in the derivation of
gradient estimates in \cite{CK}, we get
\begin{equation*}
|\langle x, \nabla u_i \rangle+\frac{\beta_i}{2\pi}|\leq
C|x|^{-\varepsilon} \qquad \text{for} \quad |x|\quad
\text{near}\quad \infty.
\end{equation*}
Consequently we have
\begin{equation*}
|u_{ir}+\frac{\beta_i}{2\pi r}|\leq C|x|^{-1-\varepsilon} \qquad
\text{for} \quad |x|\quad \text{near}\quad \infty.
\end{equation*}
Similarly, we get
\begin{equation*}
|u_{i \theta}|\leq C|x|^{-\varepsilon} \qquad \text{for} \quad
|x|\quad \text{near}\quad \infty.
\end{equation*}
Here $(r,\theta)$ is the polar coordinate system on $\R^2$ and $C,\
\varepsilon$ are  positive constants. In particular, we derive the asymptotic estimate (\ref{ayu}).

\

For $\psi_i$, since $\beta_i>2\pi$ for $i=1,2$, it follows from (\ref{asy-psi-3}) that (\ref{asy-psi-4}) holds.

\


Now we set \[
T(z)dz^2=\{\sum_{j,k=1}^2 a^{jk}((u_k)_{zz}-(u_j)_z (u_k)_z-\frac 18\langle
\psi_k,dz\cdot\partial_{\bar z}\psi_k\rangle -\frac 18\langle
d\bar{z}\cdot\partial_z\psi_k,\psi_k\rangle ) \}dz^2
\]
where $\cdot$ is the Clifford multiplication. We know from Proposition
\ref{pro-holo} that $T(z)$ is holomorphic in $\R^2$. By using the asymptotic estimate (\ref{asy-psi-3}) and (\ref{ayu}), we have the following expansion
of $T(z)$ near infinity
\begin{eqnarray*}
&& \frac 14\frac 1{z^2} \sum_{j,k=1}^2 a^{jk}(2\frac {\beta_k}{2\pi}-\frac {\beta_j}{2\pi}\frac {\beta_k}{2\pi})+o(\frac
{1}{z^2})+\cdots\\
\end{eqnarray*}
Hence, $T(z)$ is a constant and we have
$$
 \sum_{j,k=1}^2 a^{jk}(2\frac {\beta_k}{2\pi}-\frac {\beta_j}{2\pi}\frac {\beta_k}{2\pi})=0.
$$
This implies the relation between $\beta_1$ and $\beta_2$ as well as the relation between $\alpha_1$ and $\alpha_2$.
\qed

\

It is natural to ask the following question: can an entire solution $(u_1,u_2,\psi_1,\psi_2)$ of (\ref{bequ-1}) defined on $\R^2$ and satisfying (\ref{bequ-2})  extend conformally to a solution on $\S^2 ?$ This is true for the case of the super-Liouville Equation (Theorem 6.4 in \cite{JWZ}). Here we would like to propose the following:

\begin{conj} Any solution $(u_1,u_2,\psi_1,\psi_2)$ to (\ref{bequ-1}) defined on $\R^2$ and satisfying (\ref{bequ-2})  extends conformally to a smooth solution on $\S^2$. Furthermore, $\alpha_1=\alpha_2=4\pi$, and $\beta_1=\beta_2=4\pi$.
\end{conj}

\

\section{Brezis-Merle type concentration compactness}

\

In this section, we shall study the Brezis-Merle type concentration compactness for a sequence of solutions to (\ref{equ-1}) with uniformly bounded energy. We begin with a small energy regularity theorem.

\begin{lm}\label{snrlm} Let $(u_{1n},u_{2n},\psi_{1n},\psi_{2n})$ be a sequence of solutions to (\ref{equ-1})
in $B_r$. Assume that for some fixed constant $C>0$, there holds
$$
\int_{B_r}|\psi_{in}|^4dx<C,
$$
for $i=1,2$.  If there exists a sufficiently small $\varepsilon_0>0$ such that
$$
\int_{B_r}e^{2u_{in}}dx<\varepsilon_0, \qquad i=1,2,
$$
then  $\max\{u^+_{1n},u^+_{2n}\}$ is uniformly bounded in $L^\infty(B_{\frac r4})$.
\end{lm}

\begin{proof} To prove this lemma, let us define
 \begin{equation}\label{f1n}
 f_{1n}=2e^{2u_{1n}}-e^{2u_{2n}}-e^{u_{1n}}\left\langle \psi_{1n} ,\psi_{1n}
\right\rangle +\frac 12 e^{u_{2n}}\left\langle \psi_{2n} ,\psi_{2n}
\right\rangle,
\end{equation}
and
 \begin{equation}\label{f2n}
f_{2n}=2e^{2u_{2n}}-e^{2u_{1n}}-e^{u_{2n}}\left\langle \psi_{2n} ,\psi_{2n}
\right\rangle +\frac 12 e^{u_{1n}}\left\langle \psi_{1n} ,\psi_{1n}
\right\rangle.
 \end{equation}
By standard elliptic theory, it is sufficient to show that $f_{in}$ is uniformly bounded in $L^q_{loc}(B_{r})$ for some $q>1$.

Let $w_{1n}$ be the solution of the following problem
\begin{equation*}
\left\{
\begin{array}{rcll}
-\triangle w_{1n} &=& 2e^{2u_{1n}}+\frac 12 e^{u_{2n}}\left\langle \psi_{2n} ,\psi_{2n} \right\rangle &\qquad \text{in} \quad B_{r}\\
w_{1n}&=& 0, &\qquad\text{on}\quad \partial B_{r}.
\end{array}
\right. \end{equation*} It is clear that $w_{1n}\geq 0$ in $B_r$. Furthermore, we can choose $\varepsilon_0>0$ sufficiently small such that
$$
\int_{B_r}(2e^{2u_{1n}}+\frac 12 e^{u_{2n}}\left\langle \psi_{2n} ,\psi_{2n} \right\rangle)dx < 2\varepsilon_0+\frac 12\sqrt{C\varepsilon_0}<2\pi.
$$
Then it follows from (\ref{BMInq}) that there exists some $p>1$ such that
$$
e^{2w_{1n}} \text { is uniformly bounded in } L^p(B_r).
$$ Now it is clear that $\triangle (u_{1n}-w_{1n})=e^{2u_{2n}}+ e^{u_{1n}}\left\langle \psi_{1n} ,\psi_{1n} \right\rangle\geq 0 $ in $B_r$ and
$$
\int_{B_r}(u_{1n}-w_{1n})^+dx\leq \int_{B_r}u^+_{1n}dx\leq \int_{B_r}e^{u_{1n}}dx<C
$$ for some constant $C>0$. Therefore, by the mean value theorem for subharmonic functions, for any $y\in B_{\frac 12}$, we have
$$
(u_{1n}-w_{1n})(y)\leq C \int_{B_r}(u_{1n}-w_{1n})^+dx\leq C.
$$
Thus, we deduce that
$$
e^{2u_{1n}} \text { is uniformly bounded in } L^p(B_{\frac r2})
$$ for some $p>1$.
Applying a similar argument, we obtain that
$$
e^{2u_{2n}} \text { is uniformly bounded in } L^p(B_{\frac r2}),
$$ for some $p>1$. By H\"{o}lder's inequality, for $l=\frac {2p}{p+1}>1$, we have
$$
\int_{B_{\frac r2}}(e^{u_{in}}|\psi_{in}|^2)^ldx\leq (\int_{B_{\frac r2}}(e^{2pu_{in}}dx)^{\frac l{2p}}(\int_{B_{\frac r2}}|\psi_{in}|^4dx)^{\frac {2p-l}{2p}}\leq C.
$$
Let $q=\min\{l,p\}>1$,  then $f_{in}$ is uniformly bounded in $L^q(B_\frac r2)$ with $q>1$ and the conclusion of the Lemma follows.
\end{proof}

By Lemma \ref{snrlm}, we can show the Brezis-Merle type concentration compactness theorem.

\

\noindent{\bf Proof of Theorem \ref{BMTH}:}
Firstly, for $i=1,2$, by the equation
$\slashiii{D}\psi _{in}=-e^{u_{in}}\psi _{in}$,  we know that if $u_{in}$ is uniformly bounded from above then $|\psi_{in}|$ is uniformly bounded. Therefore we have  $\Sigma_{\psi_{in}}\subset\Sigma_{u_{in}}$ and $\psi _{in}$ is bounded in $
L_{loc}^\infty (M\backslash \Sigma _{\psi_{in}})$.

Next, by the energy condition, we know $e^{2u_{in}}$ is uniformly bounded in $
L^1(M)$, and so, we may extract a subsequence from $u_{in}$ (still denoted by
$u_{in}$) such that $e^{2u_{in}}$ converges in the sense of measures
on $M$ to some nonnegative bounded measure $\mu_i$, i.e.,
\begin{equation*}
\int_{M}e^{2u_{in}}\varphi dv \rightarrow \int_{M}\varphi d\mu_i
\end{equation*}
for every $\varphi \in C(M)$.

For $\varepsilon>0$, let us define an $\varepsilon$-regular point with respect to $\mu_i$. A point $x\in M$ is called an $\varepsilon-$regular point with respect to $\mu_i$ if there is a
function $\varphi \in C(M),$ satisfying $\text{supp}\varphi \subset
B_r(x)\subset M,$ $0\leq \varphi \leq 1$ and $\varphi =1$ in a
neighborhood of $x$ such that
\begin{equation*}
\int_{M}\varphi d\mu_i  <\varepsilon.
\end{equation*}

Let $\varepsilon_0>0$ be the constant as in Lemma \ref{snrlm}, we set
\begin{equation*}
M _i(\varepsilon_0) = \{x\in M: x\text{ is not an }\varepsilon_0 -\text{
regular point with respect to }\mu_i\}. \\
\end{equation*}
By the energy condition, it is clear that $M _i(\varepsilon_0 )$ is
finite for $i=1,2$.

\

On the one hand, if $x_0\notin \Sigma_{u_{in}}$ for some $i\in \{1,2\}$, then there exists $R_0>0$ such that
$$
\max_{\bar{B}_{R_0}(x_0)}u_{in}\leq C,  \qquad \text {for any } n\in \N.
$$
Therefore we have
$$
\int_{B_R(x_0)}e^{2u_{in}}dx\leq CR^{2}$$ for all $0<R<R_0$.  This implies that
$x_0$ is a regular point and $x_0\notin M_i(\varepsilon_0)$. Thus we have $M_i(\varepsilon_0)\subset \Sigma_{u_{in}}$.

On the other hand, if $x_0\notin M_1(\varepsilon_0)\cup M_2(\varepsilon_0)$, then there exists $r_0>0$ such that for $ i=1,2$
$$
\int_{B_{r_0}(x_0)}d\mu_i< \varepsilon_0.
$$ Hence, by Lemma \ref{snrlm} we have that $u^+_{in}$ is  bouned in $L^\infty(B_{\frac {r_0}2}(x_0))$ for $ i=1,2$. It follows that $x_0\notin \Sigma_{u_{1n}}\cup\Sigma_{u_{2n}}$. Thus we have $\Sigma_{u_{1n}}\cup\Sigma_{u_{2n}}\subset M_1(\varepsilon_0)\cup M_2(\varepsilon_0)$.

Therefore, we obtain that $$
\Sigma _{u_{1n}}\cup \Sigma _{u_{2n}}= M _1(\varepsilon _0)\cup M _2(\varepsilon _0),
$$ and  $\Sigma _{u_{in}}$ is a finite set for all $i=1,2$.

\

Now we can show the convergence of $u_{in}$. The crucial tool here is a Harnack type lemma as in \cite{BM, Lu}.

 If  $\Sigma _{u_{1n}}\cup \Sigma _{u_{2n}}=\emptyset $, then  $u_{in}^{+}$ is bounded in $L_{loc}^\infty(M)$ for $ i=1,2$.  Consequently, by the spinor equations, $\psi _{in}$ is bounded in $L_{loc}^\infty (M)$ for $ i=1,2$.  Let $f_{in}$ be the functions defined as in the proof of Lemma \ref{snrlm}, see \eqref{f1n} and \eqref{f2n}. Then $f_{in}$ is bounded in $L^p(M)$ for any $p>1$ and for $ i=1,2$.  Applying the Harnack inequality, we can get i) and ii) of b) .

If $\Sigma _{u_{1n}}\cup \Sigma _{u_{2n}}\neq \emptyset $, then $u_{in}^{+}$ is bounded in $L_{loc}^\infty
(M\backslash (\Sigma _{u_{1n}}\cup \Sigma _{u_{2n}})$ for $i=1,2$. By using the Harnack inequality, we can get iii) of b).
\qed

\

\section{Bubbling solutions and energy identity for spinors}

\

In this section, we will show the energy identities for the spinor parts of a sequence of blow-up solutions of the super-Toda system \eqref{equ-1} with uniformly bounded energy. To show this, we need some lemmas about the spinor equations.

Firstly, we state an $\epsilon$-regularity theorem for the Dirac type equation. Result of this type were obtained in \cite{W} for a different but similar equation and further developed in \cite{SZ, JKTWZ} in a more general setting.
\begin{lm}\label{spinor-epsilon}
For any $2\leq q<\infty$,  there exist $\epsilon = \epsilon(q)>0$ and $C=C(q)>0$ such that if a spinor $\psi\in L^4(B_1)$ is a weak solution of
\begin{equation*}
\slashiii{D} \psi = - e^{u}  \psi\quad \text{in $B_1$.}
\end{equation*}
with
$$\int_{B_1} e^{2u}\leq \epsilon, $$
then
\begin{equation*}
\|\psi\|_{L^q(B_{\frac12})} + \|\nabla \psi\|_{L^{\frac{2q}{2+q}}(B_{\frac12})} \leq C\|\psi\|_{L^{4}(B_1)}.
\end{equation*}
\end{lm}
\begin{proof}
$\psi \in L^4(B_1)$ is a weak solution of an equation of the form
\begin{equation*}
\slashiii{D} \psi = A \psi\quad \text{in $B_1$.}
\end{equation*}
where the potential $A=- e^{u} $ has small $L^2$ norm.  Then we can apply the same argument as in the proof of Theoem 3.4 in \cite{SZ} to get the conclusions. See also Lemma 5.1 in \cite{JKTWZ} for a more general equation.
\end{proof}

Consequently, we have the following energy gap result for  the Dirac type equation on $\R^2\setminus{\{0\}}$:
\begin{lm}\label{spinor-energygap}
There exists a universal constant $\epsilon_1 >0$ such that if $(u,\psi)$ solves weakly
\begin{equation} \label{dirac equation}
\slashiii{D} \psi = - e^{u}  \psi\quad \text{in $\R^2\setminus{\{0\}}$.}
\end{equation}
and satisfies
$$\int_{\R^2} e^{2u}\leq \epsilon_1, \quad  \int_{\R^2} |\psi|^4 < \infty, $$
then $\psi \equiv 0 $ on $\R^2$.
\end{lm}

\begin{proof} Let $\psi \in L^4(\R^2)$ solve \eqref{dirac equation} weakly and with $u$ satisfying $\int_{\R^2} e^{2u} < \infty.$ Then it is easy to verify that $\psi \in L^4(\R^2)$ is a weak solution on the whole $\R^2$.

By conformal invariance, for any $r>0$, define
\begin{equation*}
\left\{
\begin{array}{rcl}
\widetilde{u}(x)&=&u(r x)+ \ln r\\
\widetilde{\psi}(x)&=& r^{\frac 12} \psi (r x)\\
\end{array}
\right.
\end{equation*}
then $(\widetilde{u}, \widetilde{\psi})$ satisfies
\begin{equation*}
\slashiii{D} \widetilde{\psi} = - e^{\widetilde{u}}  \widetilde{\psi}\quad \text{in $\R^2$.}
\end{equation*}
and  the following holds:
$$ \int_{B_1} e^{2\widetilde{u}} = \int_{B_r} e^{2u}\leq\int_{\R^2} e^{2u}\leq \epsilon_1.$$

Applying Lemma \ref{spinor-epsilon} with some fixed $q>4$, if we take  $\epsilon_1=\frac{\epsilon(q)}{2}$ with $\epsilon(q)>0$ as in Lemma \ref{spinor-epsilon},  then
\begin{equation*}
 \|\widetilde{\psi}\|_{L^q(B_{\frac{1}{2}})}\leq   C(q) \|\widetilde{\psi}\|_{L^4(B_1)}.
\end{equation*}
It follows that
\begin{equation*}
 \|\psi\|_{L^q(B_{\frac{r}{2}})}\leq   C(q)     r^{\frac{4-q}{2q}}       \|\psi\|_{L^4(B_r)}  \leq   C(q)     r^{\frac{4-q}{2q}}       \|\psi\|_{L^4(\R^4)}.
\end{equation*}
Letting $r\rightarrow \infty$, we get $\psi \equiv 0 $ on $\R^2$.
\end{proof}

The next lemma is about some elliptic estimates for  spinors on an annulus, see Lemma 3.1 in \cite{JWZZ1}. It is easy to see that the Dirac type equation itself is sufficient to get these estimates.
\begin{lm}\label{main-lamm}
There exist two universal constants $\Lambda>0$ and $C>0$ such that if $(u,\psi)$ solves
\begin{equation*}
\slashiii{D} \psi = - e^{u}  \psi\quad \text{in $B_1$.}
\end{equation*}
and $(u,\psi)$ has finite energy on the annulus
$A_{r_1,r_2}=\{x\in \R^2|r_1\leq |x|\leq r_2\}$, where
$0<r_1<2r_1<\frac {r_2}2<r_2<1$, then we have
\begin{eqnarray}\label{inqu}
&&(\int_{A_{2r_1,\frac {r_2}2}}|\nabla \psi|^{\frac 43})^{\frac
34}+(\int_{A_{2r_1,\frac {r_2}2}}|\psi|^4)^{\frac 14}\\
&\leq & \Lambda (\int_{A_{r_1,r_2}}e^{2u})^{\frac
12}(\int_{A_{r_1,r_2}}|\psi|^4)^{\frac
14}+C(\int_{A_{r_1,2r_1}}|\psi|^4)^{\frac 14}+C(\int_{A_{\frac
{r_2}2, r_2}}|\psi|^4)^{\frac 14}\nonumber
\end{eqnarray}
\end{lm}

\

As is mentioned in the introduction, for a blow-up sequence of solutions to the super-Toda system with uniformly bounded energy, there possibly exist 4 types of bubbling solutions of \eqref{equ-sl} or \eqref{equ-1} defined on $\R^2$ or $\R^2\setminus \{0\}$ and with finite energy.
It turns out that the two integral quantities defined as in \eqref{entire-energy-sl} and \eqref{entire-energy-st} for solutions of \eqref{equ-1} and \eqref{equ-sl}  on $\R^2\setminus \{0\}$ are both nonnegative. This is achieved by applying standard potential analysis and the Kelvin transformation.

\begin{lm}\label{lm-5.5}
 Let $(u_1,u_2,\psi_1,\psi_2)$ be a solution to (\ref{equ-1}) in  $ B_1\setminus\{0\}$
and with finite energy
$$\int_{B_1}(e^{2u_1}+e^{2u_2}+|\psi_1|^4+|\psi_2|^4) < \infty.$$  Then there exist $s_i$  with $s_i<2\pi$ such that
\begin{eqnarray*}
u_i(x) &=& -\frac{s_i}{2\pi}\ln|x|+h_i(x)\quad \text { near }  0
\end{eqnarray*}
for $i=1,2$. Moreover, $h_i(x)$ is bounded near $0$, and satisfies
 \begin{eqnarray*}
|\nabla h_i(x)| & \leq & \frac {C}{|x|^{1-\delta_i}}\quad \text { near } 0
\end{eqnarray*}
for some constant $C>0$  and some small $\delta_i>0$ and $i=1,2$.
\end{lm}

\begin{proof}
The conclusions follow from applying a similar argument as in the proof of Proposition 4.5 in \cite{JZZ3}, where the case of super-Liouville type equations was studied. For the case of Liouville type equations, results of this kind can be found in the proof of Theorem 4  in \cite{BT}.
\end{proof}

Consequently, we have

\begin{thm}\label{nonnegative-st}
Let $(u_1,u_2,\psi_1,\psi_2)$ be a solution to (\ref{equ-1}) in  $ \R^2\setminus\{0\}$
with finite energy
$$\int_{\R^2}(e^{2u_1}+e^{2u_2}+|\psi_1|^4+|\psi_2|^4) < \infty.$$
Then we have
\begin{equation*}
\alpha_i =\int_{\R^2}e^{2u_i}-\frac 12 e^{u_i}|\psi_i|^2dx >  0,
\end{equation*}for $i=1, 2$.
\end{thm}
\begin{proof} For $i=1, 2$, by Lemma \ref{lm-5.5}, we have
\begin{eqnarray*}
u_i(x) &=& -\frac{s_{i1}}{2\pi}\ln|x|+h_{i1}(x)\quad \text { near } 0
\end{eqnarray*}
for some $s_{i1}<2\pi$. Here, $h_{i1}(x)$ is bounded near $0$, and
 \begin{eqnarray*}
|\nabla h_{i1}(x)| & \leq & \frac {C}{|x|^{1-\delta_{i1}}}\quad \text { near } 0
\end{eqnarray*}
for some constant $C>0$ and some small $\delta_{i1}>0$. Consider the Kelvin transformation of $(u_1,u_2,\psi_1,\psi_2)$
\begin{eqnarray*}
&& v_i(x)=u_i(\frac {x}{|x|^2})-2\ln |x|,\\
&& \phi_i (x)=|x|^{-1} \psi_i (\frac{x}{|x|^2}),
\end{eqnarray*}
for $i=1,2$.
By conformal invariance, $(v_1,v_2,\phi_1,\phi_2)$ satisfies \eqref{equ-1} in  $ \R^2\setminus\{0\}$
with finite energy
$$\int_{\R^2}(e^{2v_1}+e^{2v_2}+|\phi_1|^4+|\phi_2|^4) < \infty.$$
By applying Lemma \ref{lm-5.5} again, for $i=1,2$, we have
\begin{eqnarray*}
v_i(x) &=& -\frac{s_{i2}}{2\pi}\ln|x|+h_{i2}(x)\quad \text { near } 0
\end{eqnarray*}
for some $s_{i2}<2\pi$, while $h_{i2}(x)$ is bounded near $0$ and
 \begin{eqnarray*}
|\nabla h_{i2}(x)| & \leq & \frac {C}{|x|^{1-\delta_{i2}}}\quad \text { near } 0
\end{eqnarray*}
for some constant $C>0$ and some small $\delta_{i2}>0$.
Therefore, by using the Kelvin transformation again we get
\begin{eqnarray*}
u_i(x) &=& (\frac{s_{i2}}{2\pi}-2)\ln|x|+h_{i3}(x)\quad \text { near } \infty,
\end{eqnarray*}
where $h_{i3}(x)$ satisfies
 \begin{eqnarray*}
|\nabla h_{i3}(x)| & \leq & \frac {C}{|x|^{1+\delta_{i2}}}\quad \text { near } \infty.
\end{eqnarray*}
Now we integrate the equation for $u_1$ in \eqref{equ-1} on $B_R\backslash B_r$ to obtain
\begin{eqnarray*}
&& \int_{B_R\backslash B_r}2e^{2u_1}-e^{2u_2}-e^{u_1}|\psi_1|^2+\frac 12 e^{u_2}|\psi_2|^2   \\
&= & -\int_{B_R\backslash B_r}\triangle u_1            \\
&=& -\int_{\partial B_R}\frac{\partial u_1}{\partial n}+\int_{\partial B_r}\frac{\partial u_1}{\partial n}.
\end{eqnarray*}
Letting $R\rightarrow \infty$ and $r\rightarrow 0$, we have
$$
2\alpha_1-\alpha_2=2\pi(2-\frac{s_{12}+s_{11}}{2\pi})>0.
$$
Similarly,  we  have
$$
2\alpha_2-\alpha_1=2\pi(2-\frac{s_{22}+s_{21}}{2\pi})>0.
$$
Hence we obtain $\alpha_i>0$ for $i=1,2$.
\end{proof}

For super-Liouville equations, we have analogous results. Since the proofs are almost the same as for the super-Toda case, we omit them and only state
the corresponding results.
\begin{lm}\label{lm-5.7}
 Let $(u,\psi)$ be a solution to (\ref{equ-sl}) in  $ B_1\setminus\{0\}$
and with finite energy
$$\int_{B_1}(e^{2u}+|\psi|^4) < \infty.$$
Then there exists $s<2\pi$ such that
\begin{eqnarray*}
u(x) &=& -\frac{s}{2\pi}\ln|x|+h(x)\quad \text { near } 0.
\end{eqnarray*}
Moreover, $h(x)$ is bounded near $0$, and satisfies
 \begin{eqnarray*}
|\nabla h(x)| & \leq & \frac {C}{|x|^{1-\delta}}\quad \text { near } 0
\end{eqnarray*}
for some constant $C>0$ and some small $\delta>0$.
\end{lm}

\begin{thm}\label{nonnegative-sl}
Let $(u,\psi)$ be a solution to (\ref{equ-sl}) in  $ \R^2\setminus\{0\}$
with finite energy
$$\int_{\R^2}(e^{2u}+|\psi|^4) < \infty.$$
Then we have
\begin{equation*}
\alpha =\int_{\R^2}e^{2u}-\frac 12 e^{u}|\psi|^2dx >0.
\end{equation*}
\end{thm}

Finally, we are able to show the energy identities for spinors.

\

\noindent{\bf Proof of Theorem \ref{engy-indt}:}  First of all, by the energy condition and small energy regularity, it is easy to see that, by passing to subsequences if necessary, $\psi_{jn}$  converges weakly to some limit $ \psi_j$  in $L^4$  and the convergence is in $C_{loc}^{2}$ on $M\backslash (\Sigma_{1n}\cup \Sigma_{2n})$ for all $j=1,2$.

Without loss of generality, we may assume that $\Sigma_{1n}\cup \Sigma_{2n}\neq \emptyset$. Given a blow-up point $p\in \Sigma_{1n}\cup \Sigma_{2n}$, let $D_{2\delta}$ be a small ball centered at $p=0$ such that $p$ is the only blow-up point in $D_{2\delta}$. By conformal invariance, we shall rescale each $(u_{1n},u_{2n},\psi_{1n},\psi_{2n})$ near the blow-up point $p$  as in the case of super-Liouville equations. Since
$$
\max_{\overline{D}_\delta}\max\{u_{1n},u_{2n}\}\rightarrow +\infty \qquad \text { as } n\rightarrow \infty,
$$
without loss of generality, we  may also assume that there exists a sequence of points $x_n\in \overline{D}_\delta $ such that $$u_{1n}(x_n)=\max_{\overline{D}_\delta}\max\{u_{1n},u_{2n}\}.$$
Since $p$ is the only blow-up point in $D_{2\delta}$, we have $x_n\rightarrow p=0$ and
$u_{1n}(x_n)\rightarrow +\infty$. Let $\lambda_n =e^{-u_{1n}(x_n)}\rightarrow 0$ and define the rescaled fields $(\widetilde{u}_{1n},\widetilde{u}_{2n},\widetilde{\psi}_{1n},\widetilde{\psi}_{2n})$ as
\begin{equation*}
\left\{
\begin{array}{rcl}
\widetilde{u}_{1n}(x)&=&u_{1n}(\lambda_nx+x_n)+\ln {\lambda_n}\\
\widetilde{u}_{2n}(x)&=&u_{2n}(\lambda_nx+x_n)+\ln {\lambda_n}\\
\widetilde{\psi}_{1n}(x)&=&\lambda_n^{\frac 12}\psi_{1n}(\lambda_nx+x_n)\\
\widetilde{\psi}_{2n}(x)&=&\lambda_n^{\frac 12}\psi_{2n}(\lambda_nx+x_n)
\end{array}
\right.
\end{equation*}
for any $x\in \overline{D}_{\frac{\delta}{2\lambda_n}}$. Then we have
\begin{equation*}
\left\{
\begin{array}{rcl}
-\triangle \widetilde{u}_{1n}(x)&=& 2 e^{2\widetilde{u}_{1n}(x)}-e^{2\widetilde{u}_{2n}(x)}-e^{\widetilde{u}_{1n}(x)}|\widetilde{\psi}_{1n}(x)|^2+\frac 12e^{\widetilde{u}_{2n}(x)}|\widetilde{\psi}_{2n}(x)|^2, \\
-\triangle \widetilde{u}_{2n}(x)&=& 2 e^{2\widetilde{u}_{2n}(x)}-e^{2\widetilde{u}_{1n}(x)}-e^{\widetilde{u}_{2n}(x)}|\widetilde{\psi}_{2n}(x)|^2+\frac 12e^{\widetilde{u}_{1n}(x)}|\widetilde{\psi}_{1n}(x)|^2, \\
\slashiii {D}\widetilde{\psi}_{1n}(x)&=&-e^{\widetilde{u}_{1n}(x)}\widetilde{\psi}_{1n}(x),\\
\slashiii {D}\widetilde{\psi}_{2n}(x)&=&-e^{\widetilde{u}_{2n}(x)}\widetilde{\psi}_{2n}(x),
\end{array}
\right.
\end{equation*}
in $ \overline{D}_{\frac{\delta}{2\lambda_n}}$. Furthermore,
$$
\int_{D_{\frac{\delta}{2\lambda_n}}}\left (e^{2\widetilde{u}_{1n}(x)}+e^{2\widetilde{u}_{2n}(x)}+|\widetilde{\psi}_{1n}(x)|^4 +|\widetilde{\psi}_{2n}(x)|^4\right)dx<C.
$$
Noticing that $\widetilde{u}_{1n}(0)=0\geq \widetilde{u}_{2n}(0)$ and $\widetilde {u}_{in}(x)\leq 0$, from Theorem \ref{BMTH}, we have two cases:

\

\noindent {\bf Case I}: By passing to a subsequence, $(\widetilde{u}_{1n},\widetilde{u}_{2n},\widetilde{\psi}_{1n},\widetilde{\psi}_{2n})$ converges in  $C^2_{loc}(\R^2)\times C^2_{loc}(\R^2)\times  C^2_{loc}(\Gamma(\Sigma \R^2)) \times  C^2_{loc}(\Gamma(\Sigma \R^2))$ to some $(\widetilde
u_1,\widetilde u_2, \widetilde \psi_1,\widetilde \psi_2)$ satisfying
\begin{equation*}
\left\{
\begin{array}{rcl}
-\triangle \widetilde{u}_{1}(x)&=& 2 e^{2\widetilde{u}_{1}(x)}-e^{2\widetilde{u}_{2}(x)}-e^{\widetilde{u}_{1}(x)}|\widetilde{\psi}_{1}(x)|^2+\frac 12e^{\widetilde{u}_{2}(x)}|\widetilde{\psi}_{2}(x)|^2, \\
-\triangle \widetilde{u}_{2}(x)&=& 2 e^{2\widetilde{u}_{2}(x)}-e^{2\widetilde{u}_{1}(x)}-e^{\widetilde{u}_{2}(x)}|\widetilde{\psi}_{2}(x)|^2+\frac 12e^{\widetilde{u}_{1}(x)}|\widetilde{\psi}_{1}(x)|^2, \\
\slashiii {D}\widetilde{\psi}_{1}(x)&=&-e^{\widetilde{u}_{1}(x)}\widetilde{\psi}_{1}(x),\\
\slashiii {D}\widetilde{\psi}_{2}(x)&=&-e^{\widetilde{u}_{2}(x)}\widetilde{\psi}_{2}(x),
\end{array}
\right.
\end{equation*}
in $\R^2$ and
$$
\int_{\R^2}\left (e^{2\widetilde{u}_{1}(x)}+e^{2\widetilde{u}_{2}(x)}+|\widetilde{\psi}_{1}(x)|^4 +|\widetilde{\psi}_{2}(x)|^4\right)dx<C.
$$
Hence, it follows from Theorem \ref{asyth}, that
\[\int_{\R^2}(e^{2\widetilde{u}_i}-\frac 12 e^{\widetilde {u}_i}|\widetilde{\psi}_i|^2)dx > 2\pi,\] for $i=1,2$, and $(\widetilde
u_1,\widetilde u_2, \widetilde \psi_1,\widetilde \psi_2)$ is  a bubbling solution of \eqref{equ-1} on $\R^2$.

\
\

\noindent {\bf Case II}:   By passing to a subsequence, $(\widetilde{u}_{1n},\widetilde{\psi}_{1n} , \widetilde{\psi}_{2n}) $ converges in  $C^2_{loc}(\R^2)\times C^2_{loc}(\Gamma(\Sigma \R^2))$ to some $(\widetilde u_1, \widetilde \psi_1,\widetilde \psi_2)$ and $\widetilde{u}_{2n}$ converges to $-\infty $ uniformly in any compact subset in $\R^2$. It is clear that $(\widetilde u_1, \widetilde \psi_1,\widetilde \psi_2)$  satisfy
\begin{equation*}
\left\{
\begin{array}{rcl}
-\triangle \widetilde{u}_{1}(x)&=& 2 e^{2\widetilde{u}_{1}(x)}-e^{\widetilde{u}_{1}(x)}|\widetilde{\psi}_{1}(x)|^2, \\
\slashiii {D}\widetilde{\psi}_{1}(x)&=&-e^{\widetilde{u}_{1}(x)}\widetilde{\psi}_{1}(x),\\
\slashiii {D}\widetilde{\psi}_{2}(x)&=&0,
\end{array}
\right.
\end{equation*}
in $\R^2$ and
$$
\int_{\R^2}\left (e^{2\widetilde{u}_{1}(x)}+|\widetilde{\psi}_{1}(x)|^4 +|\widetilde{\psi}_{2}(x)|^4\right)dx<C.
$$
By the removability of a global singularity of a solution of a super-Liouville equation (see Proposition 2.5 in \cite{JWZZ1}), we have
\[\int_{\R^2}(e^{2\widetilde{u}_1}-\frac 12 e^{\widetilde {u}_1}|\widetilde{\psi}_1|^2)dx=2\pi,\]
and
$(\widetilde u_1,\widetilde \psi_1 )$ is  a bubbling solution of the super-Liouville equation (\ref{equ-sl}) on $\R^2$. By singularity removability for harmonic spinors with finite energy in dimension two,  $\widetilde \psi_2$ can be  conformally extended to a harmonic spinor on $S^2$. By the well-known fact that there is no nontrivial harmonic spinor on $S^2$, we have that $\widetilde \psi_2\equiv0$.

\

Then, to prove the theorem, it suffices to prove that for each fixed blow-up point  $p_i\in \Sigma_{1n}\cup \Sigma_{2n}$, there are at most finitely many bubbling solutions $(u_1^{k},u_2^k,\xi_1^{k},\xi_2^k)$ of (\ref{equ-1}) on $\R^2$ or on $\R\setminus{\{0\}}$ for $k=1,2,\cdots, K$ and bubbling solutions $(u^k_j,\phi^k_j)$ of (\ref{equ-sl}) on $\R^2$ or on $\R\setminus{\{0\}}$ for $k=1,2,\cdots, L$ and $j=1,2$  such that
\begin{equation*}
\lim_{\delta_i\rightarrow 0}\lim_{n\rightarrow
\infty}\int_{D_{\delta_i}}|\psi_{jn}|^4=\sum_{k=1}^{K}\int_{\R^2}|\xi_j^{k}|^4 + \sum_{k=1}^{L}\int_{\R^2}|\phi^{k}_j|^4,  \qquad j=1, 2
\end{equation*}

We shall follow the blow-up scheme for  super-Liouville equations (see Theorem 1.2 in \cite{JWZZ1}).  This kind of blow-up scheme was used for approximate harmonic maps in e.g. \cite{DT}. In contrast to the case of super-Liouville equations where there is only one type of bubble, namely, an entire solution of \eqref{equ-sl}  on $\R^2$ which can be conformally extended to a solution on $\S^2$, here in the super-Toda case, we have possibly four types of bubbles, i.e.,  solutions of \eqref{equ-1} or \eqref{equ-sl}  on $\R^2$ or on $\R^2 \setminus \{0\}$.

In this scheme, without loss of generality, we can assume that there is only one bubble at each blow up point $p=0\in \Sigma_{1n}\cup \Sigma_{2n}$.  The general case of multiple bubbles appearing can be handled by an induction argument (see e.g. \cite{DT}), since the total energies of $(\widetilde{u}_{1n},\widetilde{u}_{2n},\widetilde{\psi}_{1n} , \widetilde{\psi}_{2n} )$ are uniformly bounded and the energy of the scalar field for each of these four types of bubbles has a universal positive lower bound. In fact,  in this blow-up scheme, one can easily see how multiple bubbles (if any) can possibly occur and how these four types of bubbles satisfying \eqref{bubbleproperty-st} and \eqref{bubbleproperty-sl} can possibly occur.

Then, to continue our proof,   it suffices to show that the energies in the neck domain for spinors are converging to 0, i.e.,
\begin{equation}\label{e3}
\lim_{\delta\rightarrow 0}\lim_{R\rightarrow +\infty}\lim_{n\rightarrow
\infty}\int_{A_{\delta,R,n}}|\psi_{jn}|^4dx=0,  \quad  j=1,2
\end{equation}
Here the neck domain $A_{\delta, R, n}$  is
$$ A_{\delta, R, n}=\{x\in \R^2 \ | \  \lambda_n R\leq |x-x_n|\leq \delta\}.$$

\

 We shall develop the proof separately for Cases I and II.

\
\

\noindent {\bf Case I.}    First, we claim that:

\

\noindent {\bf Claim I.1}:  For any $\epsilon>0$, there is an $N>1$ such that for any $n\geq N$, we have
\begin{equation}\label{e4}
\int_{D_r(x_n)\setminus
D_{e^{-1}r}(x_n)}(|\psi_{1n}|^4+|\psi_{2n}|^4)<\epsilon, \quad
\forall r\in [e\lambda _nR, \delta].
\end{equation}

\

Now we shall prove {\bf Claim I.1}. To show (\ref{e4}), we first note the following two facts:

\

\noindent{\bf Fact I.1}: For any $\epsilon>0$ and any $T>0$, there exists some $N(T)>0$ such that for any $n\geq N(T)$ and $\delta>0$ small enough, we
have
\begin{equation}\label{2.6}
\int_{D_\delta(x_n) \setminus D_{\delta
e^{-T}}(x_n)}(|\psi_{1n}|^4+|\psi_{2n}|^4)<\epsilon.
\end{equation}

To see this, recall that $(u_{1n}, u_{2n}, \psi_{1n},\psi_{2n})$ has no blow-up point in $\overline{D}_{\delta}\backslash \{p\}$ and $\psi_{jn}$ converges strongly to $\psi_j$ in $L_{loc}^4(\overline{D}_{\delta}\backslash \{p\})$. Therefore
$$
\int_{D_\delta(x_n)\backslash D_{\delta
e^{-T}}(x_n)}|\psi_{jn}|^{4} \rightarrow \int_{D_\delta\backslash
D_{\delta e^{-T}}}|\psi_j|^{4}.
$$ Hence we can also choose $\delta>0$ small enough such that, for any given $\epsilon>0$  and any given $T>0$, there exists an $N(T)>0$ big enough such that when $n\geq N(T)$,
$$
\int_{D_\delta(x_n)\backslash D_{\delta
e^{-T}}(x_n)}|\psi_{jn}|^{4}<\frac{\epsilon}{4}.$$
for $j=1,2$. Thus we get (\ref{2.6}).

\

\noindent{\bf Fact I.2}: For any small $\epsilon>0$ and any $T>0$, we may choose an $N(T)>0$ such that when $n\geq N(T)$ and $R$ is big enough,  we have
\begin{eqnarray*}
\int_{D_{\lambda_nRe^T}(x_n)\setminus
D_{\lambda_nR}(x_n)}(|\psi_{1n}|^4+|\psi_{2n}|^4)<  \epsilon,
\end{eqnarray*}

\

To verify this, we note that
\begin{eqnarray*}
\int_{D_{\lambda_nRe^T}(x_n)\setminus
D_{\lambda_nR}(x_n)}(|\psi_{1n}|^4+|\psi_{2n}|^4)
&=& \int_{D_{Re^T}\setminus
D_{R}}(|\widetilde{\psi}_{1n}|^4+|\widetilde{\psi}_{2n}|^4)\\
&\rightarrow&  \int_{D_{Re^T}\setminus
D_{R}}(|\widetilde{\psi}_1|^4+|\widetilde{\psi}_2|^4) <  \epsilon,
\end{eqnarray*}
when $R$ is big enough.

\

Next we will use {\bf Fact I.1 and Fact I.2} to prove (\ref{e4}).  We argue  by contradiction.  Suppose that there exist $\epsilon_0>0$ and a sequence $r_n\in [e\lambda_nR,\delta]$ such that
$$
\int_{D_{r_n}(x_n)\setminus
D_{e^{-1}r_n}(x_n)}(|\psi_{1n}|^4+|\psi_{2n}|^4)\geq \epsilon_0.
$$
Then it follows from {\bf Fact I.1 and Fact I.2} that   $\frac
{\delta}{r_n}\rightarrow +\infty$ and
$\frac{\lambda_nR}{r_n}\rightarrow 0$, in particular, $r_n\rightarrow
0$  and $\frac{\lambda_n}{r_n}\rightarrow 0 $ as $n\rightarrow +\infty$.

Rescaling again, we define
\begin{equation}\label{scal-1}
\left\{
\begin{array}{rcl}
v_{1n}(x)&=& u_{1n}(r_nx+x_n)+\ln r_n, \\
v_{2n}(x)&=& u_{2n}(r_nx+x_n)+\ln r_n, \\
\varphi_{1n}(x)&=& r_n^{\frac 12}\psi_{1n}(r_nx+x_n),\\
\varphi_{2n}(x)&=& r_n^{\frac 12}\psi_{2n}(r_nx+x_n).
\end{array}
\right.
\end{equation}
Then $(v_{1n},v_{2n},\varphi_{1n},\varphi_{2n})$ satisfies
\begin{equation*}
\left\{
\begin{array}{rcl}
-\triangle v_{1n}(x) &=& 2e^{2v_{1n}(x)}-e^{2v_{2n}(x)}- e^{v_{1n}}|\varphi_{1n}|^2+\frac 12 e^{v_{2n}}|\varphi_{2n}|^2 ,\\
-\triangle v_{2n}(x) &=& 2e^{2v_{2n}(x)}-e^{2v_{1n}(x)}- e^{v_{2n}}|\varphi_{2n}|^2+\frac 12 e^{v_{1n}}|\varphi_{1n}|^2 ,\\
\slashiii{D}\varphi_{1n}(x) &=& -e^{v_{1n}(x)}\varphi_{1n}(x), \\
\slashiii{D}\varphi_{2n}(x) &=& -e^{v_{2n}(x)}\varphi_{2n}(x),
\end{array}
\right.
\end{equation*}
in $D_{\frac{\delta}{r_n}}\setminus D_{\frac{\lambda_nR}{r_n}}$. Moreover
\begin{equation}\label{3.1}
\int_{(D_1\setminus D_{e^{-1}})}(|\varphi_{1n}|^4+|\varphi_{2n}|^4)\geq \epsilon_0,
\end{equation}
and
\begin{equation*}
\int_{\R^2}(e^{2v_{1n}}+e^{2v_{2n}}+|\varphi_{1n}|^4+|\varphi_{2n}|^4)\leq C.
\end{equation*}
For  $i=1,2$, we let $\Sigma_{v_{in}}$ be the blow up set of $v_{in}$  in $\overline {D_R\setminus D_{\frac 1R}}$ for any $R>0$. Then by Theorem \ref{BMTH}, we have the following 4 cases:

\

 (1). If $\Sigma_{v_{1n}}\cup \Sigma_{v_{2n}}\neq \emptyset$, then we will obtain a second bubble, i.e. a solution of \eqref{equ-1} or \eqref{equ-sl}  on $\R^2$, by applying the rescaling argument as in the beginning of the proof. Thus we get a contradiction
to the assumption that there is only one bubble at the blow-up point $p$.

\

(2). If $\Sigma_{v_{1n}}\cup \Sigma_{v_{2n}}= \emptyset$, and for any $R>0$, $v_{in}$ tends to $-\infty$ uniformly in $\overline {D_R\setminus D_{\frac 1R}}$ for all $i=1,2$, then there are solutions $\varphi_i, i=1,2$ satisfying
\begin{equation*}
\slashiii{D}\varphi_i = 0,  \text { in } \R^2\setminus\{0\}, \quad i=1,2
\end{equation*}
with bounded  energy $||\varphi_i||_{L^4(\R^2)} < \infty$, such that
$$\lim_{n\rightarrow \infty}   ||\varphi_{in} - \varphi_i||_{L^{4}(D_R\setminus D_{\frac 1R})}=0, \quad {\rm for \ any}\ R>0,  \quad i=1,2. $$
It is clear that each $\varphi_i$ can be conformally extended to a harmonic spinor on $\S^2$, which has to be identically $0$. This will contradict (\ref{3.1}).

\

(3). If $\Sigma_{v_{1n}}\cup \Sigma_{v_{2n}}= \emptyset$, and for any $R>0$, $v_{1n},v_{2n}$ are all uniformly bounded in $\overline {D_R\setminus D_{\frac 1R}}$,
then there is a solution $(v_1,v_2,\varphi_1,\varphi_2)$ satisfying
\begin{equation*}
\left\{
\begin{array}{rcl}
-\triangle v_1 &=& 2e^{2v_1}-e^{2v_2}-e^{v_1}|\varphi_1|^2+\frac 12 e^{v_2}|\varphi_2|^2,  \\
-\triangle v_2 &=& 2e^{2v_2}-e^{2v_1}-e^{v_2}|\varphi_2|^2+\frac 12 e^{v_1}|\varphi_1|^2, \\
\slashiii{D}\varphi_1 &=& -e^{v_1}\varphi_1, \\
\slashiii{D}\varphi_2 &=& -e^{v_2}\varphi_2,
\end{array}
\right.
\end{equation*}
in $ \R^2\setminus\{0\}$ and with finite energy
$\int_{\R^2}(e^{2v_1}+e^{2v_2}+|\varphi_1|^4+|\varphi_2|^4) < \infty$, such that
$$\lim_{n\rightarrow \infty}  \left (  ||v_{in} - v_i ||_{C^{2}(D_R\setminus D_{\frac 1R})} +
 ||\varphi_{in} - \varphi_i||_{C^{2}(D_R\setminus D_{\frac 1R})} \right ) =0,
$$ for any $R>0$ and $i=1,2$.

By the fact \eqref{3.1},  we know that  $\varphi_1$ and $\varphi_2$ cannot all be vanishing. So by Lemma \ref{spinor-energygap},   $$\int_{\R^2}e^{2v_1}+e^{2v_2} \geq \epsilon_1$$
and hence $(v_1,v_2,\varphi_1,\varphi_2)$ is a bubbling solution on $ \R^2\setminus\{0\}$. Moreover, by Theorem \ref{nonnegative-st},
\begin{equation*}
\int_{\R^2}e^{2v_i}-\frac 12 e^{v_i}|\psi_i|^2dx >  0,
\end{equation*}for $i=1, 2$.
Thus we get a contradiction to the assumption that there is only one bubble at the point $p$.

\

(4). If $\Sigma_{v_{1n}}\cup \Sigma_{v_{2n}}= \emptyset$, and for any $R>0$, $v_{in}$ is  uniformly bounded in $\overline {D_R\setminus D_{\frac 1R}}$ while $v_{jn}$ ($j\neq i$) tends to $-\infty$ uniformly in $\overline {D_R\setminus D_{\frac 1R}}$ ,
then there is a solution $(v_i,\varphi_i)$ satisfying in $ \R^2\setminus\{0\}$
\begin{equation*}
\left\{
\begin{array}{rcl}
-\triangle v_i &=& 2e^{2v_i}-e^{v_i}|\varphi_i|^2  \\
\slashiii{D}\varphi_i &=& -e^{v_i}\varphi_i, \\
\slashiii{D}\varphi_j &=& 0,
\end{array}
\right.
\end{equation*}
with finite energy
$\int_{\R^2}e^{2v_i}+|\varphi_i|^4+|\varphi_j|^4< \infty$, such that
$$\lim_{n\rightarrow \infty}  \left (  ||v_{in} - v_i ||_{C^{2}(D_R\setminus D_{\frac 1R})} +
 ||\varphi_{in} - \varphi_i||_{C^{2}(D_R\setminus D_{\frac 1R})} +
 ||\varphi_{jn} - \varphi_j||_{C^{2}(D_R\setminus D_{\frac 1R})} \right ) =0,
$$ for any $R>0$. It is easy to see that $\varphi_j$ can be conformally extended  to a harmonic spinor on $S^2$ and hence $\varphi_j\equiv 0$. By \eqref{3.1},  we know that $\varphi_i$ cannot be vanishing. So by Lemma \ref{spinor-energygap},
$$\int_{\R^2}e^{2v_i}\geq \epsilon_1$$
and hence $(v_i, \varphi_i)$ is a bubbling solution on $ \R^2\setminus\{0\}$. Moreover, by Theorem \ref{nonnegative-sl},
\begin{equation*}
\int_{\R^2}e^{2v_i}-\frac 12 e^{v_i}|\psi_i|^2dx >  0.
\end{equation*}
This is a contradiction to the assumption that there is only one bubble at the point $p$.

\

Next,  we make the following claim:

\

\noindent {\bf Claim I.2}:  We can separate $A_{\delta, R, n}$ into
finitely many parts
$$
A_{\delta, R, n}=\bigcup_{k=1}^{N_k}A_k$$ such that on each part
\begin{equation}\label{e5}
\int_{A_k}e^{2u_{1n}}+e^{2u_{2n}}\leq \frac{1}{4\Lambda^2}, \quad k=1,2,\cdots, N_k.
\end{equation}
Here  $N_k\leq N_0$ with $N_0$ being an uniform integer for all
$n$ large enough, $A_k=D_{r^{k-1}}\setminus D_{r^k}$, $ \lambda _nR=r^{N_k}<...<r^k<r^{k-1}<...< r^1<\delta=r^0 ,
$ and $\Lambda$ is the constant as in Lemma \ref{main-lamm}.

\

Since the energies of  $(u_{1n}, u_{2n})$ are uniformly bounded, the proof of the above claim is now standard, see the case of super-Liouville equations in \cite{JWZZ1} (page 308). Here we omit the details of the proof.

\

Now using {\bf Claim I.1} and {\bf Claim I.2}, we can show (\ref{e3}). The arguments are similar to the case of super-Liouville equations in \cite{JWZZ1}. For the sake of completeness, we provide the details.

For $i=1,2$ and for any small $0<\epsilon<1$, let $\delta$ be small enough and  $R$ and $ n$ be large enough. We apply Lemma \ref{main-lamm} to $(u_{in},\psi_{in})$ on each part $A_l$  and use \eqref{e5} to calculate
\begin{eqnarray*}
(\int_{A_l}|\psi_n|^4)^{\frac 14} &\leq & \Lambda
(\int_{D_{er^{l-1}}\setminus D_{e^{-1}r^l}}e^{2u_n})^{\frac
12}(\int_{D_{er^{l-1}}\setminus D_{e^{-1}r^l}}|\psi_n|^4)^{\frac
14}\\
&&+C(\int_{D_{er^{l-1}}\setminus D_{r^{l-1}}}|\psi_n|^4)^{\frac
14}+C(\int_{D_{r^{l}}\setminus D_{e^{-1}r^l}}|\psi_n|^4)^{\frac 14}\\
&\leq & \Lambda ((\int_{A_l}e^{2u_n})^{\frac 12}+\epsilon^{\frac
12}+\epsilon^{\frac 12})((\int_{A_l}|\psi_n|^4)^{\frac
14}+\epsilon^{\frac 14}+\epsilon^{\frac 14})+C\epsilon^{\frac 14}\\
&\leq &\Lambda (\int_{A_l}e^{2u_n})^{\frac
12}(\int_{A_l}|\psi_n|^4)^{\frac 14}+C(\epsilon^{\frac
14}+\epsilon^{\frac 12}+\epsilon^{\frac 34})\\
& \leq & \frac 12 (\int_{A_l}|\psi_n|^4)^{\frac
14}+C \epsilon^{\frac 14},
\end{eqnarray*}
which implies that
\begin{equation}\label{2.1} (\int_{A_l}|\psi_n|^4)^{\frac 14}\leq C \epsilon^{\frac 14}.
\end{equation}
Then, using Lemma \ref{main-lamm}, \eqref{e5}, (\ref{2.1}) and applying similar arguments, we have
\begin{equation}\label{2.2}
(\int_{A_l}|\nabla\psi_n|^{\frac 43})^{\frac 34}\leq
C\epsilon^{\frac 14}.
\end{equation}
Adding (\ref{2.1}) and (\ref{2.2}) on $A_l$, we conclude that
\begin{equation}\label{2.3}
\int_{A_{\delta,
R,n}}|\psi_n|^4+\int_{A_{\delta,R,n}}|\nabla\psi_n|^{\frac
43}=\sum_{l=1}^{N_0}\int_{A_l}|\psi_n|^4+|\nabla\psi_n|^{\frac
43}\leq C\epsilon^{\frac 13}.
\end{equation}
This proves (\ref{e3}) and finishes the proof of theorem in this case.

\

{\bf Case II.}  In this case, by applying the same arguments as in Case I., we can get (\ref{e3}). Here we omit the details of the proof.
\qed

\

\section{ Improved Brezis-Merle type concentration compactness}

\

In this section, by using the fact that the energies of the spinors on neck domains are converging to $0$, we can improve the Brezis-Merle type concentration compactness in Theorem \ref{BMTH}.


\

 \noindent{\bf Proof of Theorem \ref{Pblowupb}:} In view of Theorem \ref{BMTH}, we shall argue by contradiction. Assume that the conclusion is false, i.e.  both $u_{1n}$ and $u_{2n}$ are uniformly bounded on any compact subset of $M\backslash (\Sigma_{u_{1n}}\cup  \Sigma_{u_{2n}})$. Then we know that $(u_{1n}, u_{2n}, \psi_{1n}, \psi_{2n})$ converges strongly on any compact subset of
$M\backslash   (\Sigma_{u_{1n}}\cup  \Sigma_{u_{2n}})$ to some limit solution $(u_1,u_2,\psi_1,\psi_2)$ of \eqref{equ-1} in $M\backslash   (\Sigma_{u_{1n}}\cup  \Sigma_{u_{2n}})$
with finite energy
$$\int_{M}(e^{2u_1}+e^{2u_2}+|\psi_1|^4+|\psi_2|^4)\leq C.$$

Now, let $x_0\in \Sigma_{u_{1n}}\cup \Sigma_{u_{2n}}$ and $R>0$ small so that $x_0$ is the only point of $\Sigma_{u_{1n}}\cup \Sigma_{u_{2n}}$ in $\bar{B}_R(x_0)$. We can assume that $u_{in}$ and $|\psi_{in}|$ are uniformly bounded in $L^{\infty}(\partial B_R(x_0))$ for $i=1,2$. Consequently, we have
$$
\left|\frac 23u_{1n}+\frac 13u_{2n}\right|_{L^{\infty}(\partial B_R(x_0))}\leq C.
$$
for some uniform constant $C>0$.

As  in the proof of Theorem \ref{engy-indt}, we rescale $(u_{1n},u_{2n},\psi_{1n},\psi_{2n})$ near $x_0$.
Without loss of generality, we choose $x_n\in B_{R}(x_0)$ such that
$$u_{1n}(x_n)=\max_{\bar{B}_{R}(x_0)}\max \{u_{1n}(x),u_{2n}(x)\}.$$
Then we have $x_n\rightarrow x_0$
and
$u_{1n}(x_n)\rightarrow +\infty$. Let $\lambda_n =e^{-u_{1n}(x_n)}\rightarrow 0$. Denote
\begin{equation*}
\left\{
\begin{array}{rcl}
\widetilde{u}_{1n}(x)&=&u_{1n}(\lambda_nx+x_n)+\ln {\lambda_n}\\
\widetilde{u}_{2n}(x)&=&u_{2n}(\lambda_nx+x_n)+\ln {\lambda_n}\\
\widetilde{\psi}_{1n}(x)&=&\lambda^{\frac 12}_n\psi_{1n}(\lambda_n x+x_n)\\
\widetilde{\psi}_{2n}(x)&=&\lambda^{\frac 12}_n\psi_{2n}(\lambda_nx+x_n)
\end{array}
\right.
\end{equation*}
for any $x\in B_{\frac {R}{2\lambda_n}}(0)$. Then, we consider the convergence of $(\widetilde{u}_{1n},\widetilde{u}_{2n},\widetilde{\psi}_{1n},\widetilde{\psi}_{2n})$.

\

\noindent {\bf Case I:} By passing to a subsequence, $(\widetilde{u}_{1n},\widetilde{u}_{2n},\widetilde{\psi}_{1n},\widetilde{\psi}_{2n})$ converges in  $C^2_{loc}(\R^2)\times C^2_{loc}(\R^2)\times  C^2_{loc}(\Gamma(\Sigma \R^2)) \times  C^2_{loc}(\Gamma(\Sigma \R^2))$ to some $(\widetilde
u_1,\widetilde u_2, \widetilde \psi_1,\widetilde \psi_2)$ with
\begin{eqnarray}\label{int-value-st}
\int_{\R^2}(e^{2\widetilde{u}_i}-\frac 12 e^{\widetilde {u}_i}|\widetilde{\psi}_i|^2)>2\pi,
\end{eqnarray}
for $i=1,2$, and $(\widetilde u_1,\widetilde u_2, \widetilde \psi_1,\widetilde \psi_2)$ is  a bubbling solution of \eqref{equ-1} on $\R^2$.

\
\

\noindent  {\bf Case II:}   By passing to a subsequence, $(\widetilde{u}_{1n},\widetilde{\psi}_{1n} , \widetilde{\psi}_{2n}) $ converges in  $C^2_{loc}(\R^2)\times C^2_{loc}(\Gamma(\Sigma \R^2))$ to some $(\widetilde u_1, \widetilde \psi_1,\widetilde \psi_2)$ and $\widetilde{u}_{2n}$ converges to $-\infty $ uniformly in any compact subset in $\R^2$. Moreover
\begin{eqnarray}\label{int-value-sl}
\int_{\R^2}(e^{2\widetilde{u}_1}-\frac 12 e^{\widetilde {u}_1}|\widetilde{\psi}_1|^2) =2\pi,
\end{eqnarray}
and $(\widetilde u_1,\widetilde \psi_1 )$ is  a bubbling solution of the super-Liouville equation \eqref{equ-sl} on $\R^2$.

\

By the super-Toda system  \eqref{equ-1}, we have
\begin{equation}
-\triangle  (\frac 23u_{1n}+\frac 13u_{2n} ) =   e^{2u_{1n}}-\frac{1}{2}e^{u_{1n}}|\psi_{1n}|^2-R_g,  \quad \text { in } B_R(x_0).
\end{equation}
Let $w_n$ satisfy
\begin{equation*}
\left\{
\begin{array}{rlll}
-\triangle w_n &=& e^{2u_{1n}}-\frac 12 e^{u_{1n}}|\psi_{1n}|^2-R_g, & ~~~\text { in } B_R(x_0),\\
w_n &=& -C, &~~~ \text { on } \partial B_R(x_0).
\end{array}
\right.
\end{equation*}
Then, by the maximum principle,  we have $w_n\leq \frac 23u_{1n}+\frac 13u_{2n}$ and
\begin{equation}\label{6.1}
\int_{B_R(x_0)}e^{w_{n}}\leq \int_{B_R(x_0)}e^{\frac 23u_{1n}+\frac 13u_{2n}}\leq C.
\end{equation}
On the other hand, by using similar arguments as in \cite{BM}, we have $w_n\rightarrow w$ uniformly on any compact subsets of $B_R(x_0)\backslash\{x_0\}$ and $w$ satisfies
\begin{equation*}
\left\{
\begin{array}{rlll}
-\triangle w &=& \mu, & ~~~\text { in } B_R(x_0),\\
w&=& -C, &~~~ \text { on } \partial B_R(x_0).
\end{array}
\right.
\end{equation*}

\

Then for $\delta\in(0, R)$ small enough, $L>0$ large enough  and $n$ large enough, we have
\begin{eqnarray}\label{almost2pi}
&& \int_{B_{\delta}(x_0)}(e^{2u_{1n}}-\frac 12 e^{u_{1n}}|\psi_{1n}|^2-R_g)   \\
&=&\int_{B_{\lambda_n L}(x_n)}(e^{2u_{1n}}-\frac 12 e^{u_{1n}}|\psi_{1n}|^2)+ \int_{B_\delta(x_0)\backslash B_{\lambda_nL}(x_n)}(e^{2u_{1n}}-\frac 12 e^{u_{1n}}|\psi_{1n}|^2) - \int_{B_{\delta}(x_0)}R_g \nonumber \\
&\geq & \int_{B_{L}(0)}(e^{2\widetilde{u}_{1n}}-\frac 12 e^{\widetilde{u}_{1n}}|\widetilde{\psi}_{1n}|^2)-\frac 12 \int_{B_\delta(x_0)\backslash B_{\lambda_nL}(x_n)}e^{u_{1n}}|\psi_{1n}|^2- \int_{B_{\delta}(x_0)}R_g \nonumber
\end{eqnarray}

Let $n\rightarrow \infty$ in (\ref{almost2pi}), By \eqref{int-value-st} and  \eqref{int-value-sl}, we can apply the fact from Theorem \ref{engy-indt} that the neck energy of the spinor field $\psi_{1n}$ is converging to zero (when multiple bubbles occur, we need to decompose $B_\delta(x_0)\backslash B_{\lambda_nL}$ further into bubble domains and neck domains and use the fact that  the integral quantities defined in \eqref{entire-energy-sl} and \eqref{entire-energy-st} for bubbling solutions of \eqref{equ-1} and \eqref{equ-sl}  on $\R^2\setminus \{0\}$ are all nonnegative, see Theorem \ref{nonnegative-st} and  Theorem \ref{nonnegative-sl}) to obtain
$$
\lim_{n\rightarrow\infty}\int_{B_{\delta}(x_0)}(e^{u_{1n}}-\frac 12 e^{u_{1n}}|\psi_{1n}|^2-R_g)\geq 2\pi+o(1),
$$
where $ o(1)\rightarrow 0$ as $L\rightarrow \infty$ and $\delta\rightarrow 0$.  This implies that $\mu\{x_0\}\geq 2\pi$ and $\mu\geq 2\pi \delta_{x_0}$, in particular we have
$$
w(x)\geq \log\frac{1}{|x-x_0|}+O(1), ~~~\text { as } x\rightarrow x_0 .$$
Thus we obtain
$$
\int_{B_R(x_0)}e^{2w}\geq \int_{B_\delta(x_0)}\frac{C}{|x-x_0|^2}=\infty.
$$
However, from (\ref{6.1}) and by Fatou's Lemma we have $\int_{B_R(x_0)}e^{2w}<\infty$.
Thus we get a contradiction. By Theorem \ref{BMTH}, we know that at least one of  $u_{1n}$ and $u_{2n}$ tends to $-\infty$ uniformly in any compact subset of $M\backslash   (\Sigma_{u_{1n}}\cup  \Sigma_{u_{2n}})$.

By \eqref{almost2pi} and the subsequent arguments, it is easy to see that  $m_1(p) > 2\pi, m_2(p) > 2\pi$ by \eqref{int-value-st} in {\bf Case I} and $m_1(p)\geq 2\pi, m_2(p)\geq 0$ by \eqref{int-value-sl} in {\bf Case II}.

Similarly, if we rescale $(u_{1n},u_{2n},\psi_{1n},\psi_{2n})$ near $x_0$ and choose $x_n\in B_{R}(x_0)$ such that
$$u_{2n}(x_n)=\max_{\bar{B}_{R}(x_0)}\max \{u_{1n}(x),u_{2n}(x)\}.$$
Then, we will have $m_1(p) > 2\pi, m_2(p) > 2\pi$ or  $m_2(p)\geq 2\pi, m_1(p)\geq0$.  This completes the proof.
\qed

\

\end{document}